\documentclass[12pt]{amsart}
\usepackage{amsmath}
\usepackage{amsxtra}
\usepackage{amscd}
\usepackage{amsthm}
\usepackage{amsfonts}
\usepackage{amssymb}%,mathabx}
\usepackage{eucal}
\usepackage{mathrsfs}
\usepackage{verbatim}
\usepackage[all]{xy}
\usepackage{imakeidx}
\makeindex

\usepackage{color}
\definecolor{deepjunglegreen}{rgb}{0.0, 0.29, 0.29}
\definecolor{darkspringgreen}{rgb}{0.09, 0.45, 0.27}
\definecolor{Red}{rgb}{0.7, 0,0}

\makeatletter
\usepackage{etex,etoolbox,needspace}
\pretocmd\section{\Needspace*{4\baselineskip}}{}{}
\makeatletter

\usepackage[pdftex,bookmarks=false,colorlinks=true,citecolor=darkspringgreen,debug=true,
  naturalnames=true,pdfnewwindow=true]{hyperref}

\usepackage{tikz-cd}
\usepackage{tikz}
\usetikzlibrary{arrows,decorations.pathmorphing,backgrounds,positioning,fit,petri}
\usetikzlibrary{decorations.markings}
\usetikzlibrary{patterns,calc}
\usetikzlibrary{matrix,positioning,decorations.pathreplacing}

\newtheorem{thm}{Theorem}[subsection]

\newtheorem{lem}[thm]{Lemma}
\newtheorem{prop}[thm]{Proposition}

\theoremstyle{definition}

\theoremstyle{remark}
\newtheorem{rem}[thm]{Remark}
\newcommand{\nc}{\newcommand}
\nc{\renc}{\renewcommand} \nc{\ssec}{\subsection}
\nc{\sssec}{\subsubsection}

\DeclareSymbolFont{symbolsC}{U}{pxsyc}{m}{n}
\SetSymbolFont{symbolsC}{bold}{U}{pxsyc}{bx}{n}
\DeclareFontSubstitution{U}{pxsyc}{m}{n}
\DeclareMathSymbol{\medcirc}{\mathbin}{symbolsC}{7}

\nc{\on}{\operatorname} \nc{\wh}{\widehat}
\nc\ol{\overline} \nc\ul{\underline} \nc\wt{\widetilde}

\newcommand{\red}[1]{{\color{Red}#1}}

\nc{\BA}{{\mathbb{A}}} \nc{\BC}{{\mathbb{C}}} \nc{\BF}{{\mathbb{F}}}
\nc{\BD}{{\mathbb{D}}} \nc{\BG}{{\mathbb{G}}} \nc{\BH}{{\mathbb{H}}} \nc{\BQ}{{\mathbb{Q}}}
\nc{\BM}{{\mathbb{M}}} \nc{\BN}{{\mathbb{N}}} \nc{\BO}{{\mathbb{O}}}
\nc{\BP}{{\mathbb{P}}} \nc{\BR}{{\mathbb{R}}}
\nc{\BZ}{{\mathbb{Z}}} \nc{\BS}{{\mathbb{S}}} \nc{\BW}{{\mathbb{W}}}

\nc{\CA}{{\mathcal{A}}} \nc{\CB}{{\mathcal{B}}} \nc{\CalC}{{\mathcal{C}}} \nc{\CalD}{{\mathcal{D}}}
\nc{\CE}{{\mathcal{E}}} \nc{\CF}{{\mathcal{F}}}
\nc{\CG}{{\mathcal{G}}} \nc{\CH}{{\mathcal{H}}}
\nc{\CI}{{\mathcal{I}}} \nc{\CK}{{\mathcal{K}}} \nc{\CL}{{\mathcal{L}}}
\nc{\CM}{{\mathcal{M}}} \nc{\CN}{{\mathcal{N}}}
\nc{\CO}{{\mathcal{O}}} \nc{\CP}{{\mathcal{P}}}
\nc{\CQ}{{\mathcal{Q}}} \nc{\CR}{{\mathcal{R}}}
\nc{\CS}{{\mathcal{S}}} \nc{\CT}{{\mathcal{T}}}
\nc{\CU}{{\mathcal{U}}} \nc{\CV}{{\mathcal{V}}}  \nc{\CX}{{\mathcal X}}
\nc{\CY}{{\mathcal Y}} \nc{\CW}{{\mathcal{W}}} \nc{\CZ}{{\mathcal{Z}}}

\nc{\cM}{{\check{\mathcal M}}{}} \nc{\csM}{{\check{\mathcal A}}{}}
\nc{\oM}{{\overset{\circ}{\mathcal M}}{}}
\nc{\obM}{{\overset{\circ}{\mathbf M}}{}}
\nc{\oCA}{{\overset{\circ}{\mathcal A}}{}}
\nc{\obA}{{\overset{\circ}{\mathbf A}}{}}
\nc{\ooM}{{\overset{\circ}{M}}{}}
\nc{\osM}{{\overset{\circ}{\mathsf M}}{}}
\nc{\vM}{{\overset{\bullet}{\mathcal M}}{}}
\nc{\nM}{{\underset{\bullet}{\mathcal M}}{}}
\nc{\oD}{{\overset{\circ}{\mathcal D}}{}}
\nc{\obD}{{\overset{\circ}{\mathbf D}}{}}
\nc{\op}{{\overset{\bullet}{\mathbf p}}{}}
\nc{\cp}{{\overset{\circ}{\mathbf p}}{}}
\nc{\oU}{{\overset{\bullet}{\mathcal U}}{}}
\nc{\ofZ}{{\overset{\circ}{\mathfrak Z}}{}}

\nc{\fa}{{\mathfrak{a}}} \nc{\fb}{{\mathfrak{b}}}
\nc{\fd}{{\mathfrak{d}}} \nc{\fe}{{\mathfrak{e}}} \nc{\ff}{{\mathfrak{f}}}
\nc{\fg}{{\mathfrak{g}}} \nc{\fgl}{{\mathfrak{gl}}}
\nc{\fh}{{\mathfrak{h}}} \nc{\fri}{{\mathfrak{i}}}
\nc{\fj}{{\mathfrak{j}}} \nc{\fk}{{\mathfrak{k}}} \nc{\fl}{{\mathfrak{l}}}
\nc{\fm}{{\mathfrak{m}}} \nc{\fn}{{\mathfrak{n}}}
\nc{\ft}{{\mathfrak{t}}} \nc{\fu}{{\mathfrak{u}}} \nc{\fv}{{\mathfrak{v}}}
\nc{\fw}{{\mathfrak{w}}} \nc{\fz}{{\mathfrak{z}}}
\nc{\fp}{{\mathfrak{p}}} \nc{\fq}{{\mathfrak{q}}} \nc{\frr}{{\mathfrak{r}}}
\nc{\fs}{{\mathfrak{s}}} \nc{\fsl}{{\mathfrak{sl}}}
\nc{\fso}{{\mathfrak{so}}} \nc{\fsp}{{\mathfrak{sp}}} \nc{\osp}{{\mathfrak{osp}}}
\nc{\hsl}{{\widehat{\mathfrak{sl}}}}
\nc{\hgl}{{\widehat{\mathfrak{gl}}}}
\nc{\hg}{{\widehat{\mathfrak{g}}}}
\nc{\chg}{{\widehat{\mathfrak{g}}}{}^\vee}
\nc{\hn}{{\widehat{\mathfrak{n}}}}
\nc{\chn}{{\widehat{\mathfrak{n}}}{}^\vee}

\nc{\fA}{{\mathfrak{A}}} \nc{\fB}{{\mathfrak{B}}} \nc{\fC}{{\mathfrak{C}}}
\nc{\fD}{{\mathfrak{D}}} \nc{\fE}{{\mathfrak{E}}}
\nc{\fF}{{\mathfrak{F}}} \nc{\fG}{{\mathfrak{G}}} \nc{\fH}{{\mathfrak{H}}}
\nc{\fI}{{\mathfrak{I}}} \nc{\fJ}{{\mathfrak{J}}}
\nc{\fK}{{\mathfrak{K}}} \nc{\fL}{{\mathfrak{L}}}
\nc{\fM}{{\mathfrak{M}}} \nc{\fN}{{\mathfrak{N}}}
\nc{\fP}{{\mathfrak{P}}} \nc{\fQ}{{\mathfrak{Q}}}
\nc{\fS}{{\mathfrak{S}}} \nc{\fT}{{\mathfrak{T}}} \nc{\fU}{{\mathfrak{U}}}
\nc{\fV}{{\mathfrak{V}}} \nc{\fW}{{\mathfrak{W}}}
\nc{\fX}{{\mathfrak{X}}} \nc{\fY}{{\mathfrak{Y}}}
\nc{\fZ}{{\mathfrak{Z}}}

\nc{\ba}{{\mathbf{a}}} \nc{\bk}{{\mathbf{k}}} \nc{\bl}{{\mathbf{l}}}
\nc{\bb}{{\mathbf{b}}} \nc{\bc}{{\mathbf{c}}} \nc{\be}{{\mathbf{e}}}
\nc{\bg}{{\mathbf{g}}} \nc{\bi}{{\mathbf{i}}} \nc{\bj}{{\mathbf{j}}} \nc{\bm}{{\mathbf{m}}}
\nc{\bn}{{\mathbf{n}}} \nc{\bp}{{\mathbf{p}}}
\nc{\bq}{{\mathbf{q}}} \nc{\br}{{\mathbf{r}}} \nc{\bt}{{\mathbf{t}}}
\nc{\bfu}{{\mathbf{u}}} \nc{\bv}{{\mathbf{v}}}
\nc{\bx}{{\mathbf{x}}} \nc{\by}{{\mathbf{y}}} \nc{\bz}{{\mathbf{z}}}
\nc{\bw}{{\mathbf{w}}} \nc{\bA}{{\mathbf{A}}}
\nc{\bB}{{\mathbf{B}}} \nc{\bC}{{\mathbf{C}}} \nc{\bE}{{\mathbf{E}}}
\nc{\bD}{{\mathbf{D}}} \nc{\bF}{{\mathbf{F}}} \nc{\bG}{{\mathbf{G}}}
\nc{\bH}{{\mathbf{H}}} \nc{\bI}{{\mathbf{I}}} \nc{\bJ}{{\mathbf{J}}}
\nc{\bK}{{\mathbf{K}}} \nc{\bL}{{\mathbf{L}}} \nc{\bM}{{\mathbf{M}}}
\nc{\bN}{{\mathbf{N}}}
\nc{\bO}{{\mathbf{O}}} \nc{\bS}{{\mathbf{S}}} \nc{\bT}{{\mathbf{T}}}
\nc{\bU}{{\mathbf{U}}} \nc{\bV}{{\mathbf{V}}} \nc{\bW}{{\mathbf{W}}}
\nc{\bX}{{\mathbf{X}}}
\nc{\bY}{{\mathbf{Y}}} \nc{\bP}{{\mathbf{P}}}
\nc{\bZ}{{\mathbf{Z}}} \nc{\bh}{{\mathbf{h}}}

\nc{\sA}{{\mathsf{A}}} \nc{\sB}{{\mathsf{B}}}
\nc{\sC}{{\mathsf{C}}} \nc{\sD}{{\mathsf{D}}}
\nc{\sE}{{\mathsf{E}}} \nc{\sF}{{\mathsf{F}}} \nc{\sG}{{\mathsf{G}}} \nc{\sH}{{\mathsf{H}}}
\nc{\sI}{{\mathsf{I}}} \nc{\sK}{{\mathsf{K}}} \nc{\sL}{{\mathsf{L}}}
\nc{\sfm}{{\mathsf{m}}} \nc{\sM}{{\mathsf{M}}} \nc{\sN}{{\mathsf{N}}}
\nc{\sO}{{\mathsf{O}}} \nc{\sQ}{{\mathsf{Q}}} \nc{\sP}{{\mathsf{P}}}
\nc{\sT}{{\mathsf{T}}} \nc{\sZ}{{\mathsf{Z}}}
\nc{\sV}{{\mathsf{V}}} \nc{\sW}{{\mathsf{W}}} \nc{\sg}{{\mathsf{g}}}
\nc{\sfp}{{\mathsf{p}}} \nc{\sq}{{\mathsf{q}}} \nc{\sr}{{\mathsf{r}}}
\nc{\sfs}{{\mathsf{s}}} \nc{\st}{{\mathsf{t}}} \nc{\sfb}{{\mathsf{b}}}
\nc{\sfc}{{\mathsf{c}}} \nc{\sd}{{\mathsf{d}}}
\nc{\sz}{{\mathsf{z}}}

\nc{\tA}{{\widetilde{\mathbf{A}}}}
\nc{\tB}{{\widetilde{\mathcal{B}}}}
\nc{\tg}{{\widetilde{\mathfrak{g}}}} \nc{\tG}{{\widetilde{G}}}
\nc{\TM}{{\widetilde{\mathbb{M}}}{}}
\nc{\tO}{{\widetilde{\mathsf{O}}}{}}
\nc{\tU}{{\widetilde{\mathfrak{U}}}{}} \nc{\TZ}{{\tilde{Z}}}
\nc{\tx}{{\tilde{x}}} \nc{\tbv}{{\tilde{\bv}}}
\nc{\tfP}{{\widetilde{\mathfrak{P}}}{}} \nc{\tz}{{\tilde{\zeta}}}
\nc{\tmu}{{\tilde{\mu}}}

\nc{\urho}{\underline{\rho}} \nc{\uB}{\underline{B}}
\nc{\uC}{{\underline{\mathbb{C}}}} \nc{\ui}{\underline{i}}
\nc{\uj}{\underline{j}} \nc{\ofP}{{\overline{\mathfrak{P}}}}
\nc{\oB}{{\overline{\mathcal{B}}}}
\nc{\og}{{\overline{\mathfrak{g}}}} \nc{\oI}{{\overline{I}}}

\nc{\eps}{\varepsilon} \nc{\hrho}{{\hat{\rho}}} \nc{\balpha}{{\boldsymbol{\alpha}}}
\nc{\bgamma}{{\boldsymbol{\gamma}}}
\nc{\blambda}{{\boldsymbol{\lambda}}} \nc{\bmu}{{\boldsymbol{\mu}}} \nc{\bnu}{{\boldsymbol{\nu}}}
\nc{\btheta}{{\boldsymbol{\theta}}} \nc{\bzeta}{{\boldsymbol{\zeta}}} \nc{\bta}{{\boldsymbol{\eta}}}
\nc{\bbeta}{{\boldsymbol{\beta}}} \nc{\bkappa}{{\boldsymbol{\kappa}}} \nc{\bomega}{{\boldsymbol{\omega}}}

\nc{\one}{{\mathbf{1}}} \nc{\two}{{\mathbf{t}}}

\nc\DMO\DeclareMathOperator
\DMO\Sym{Sym}%\nc{\Sym}{\mathop{\operatorname{\rm Sym}}}
\nc{\Tot}{{\mathop{\operatorname{\rm Tot}}}}
\nc{\Spec}{\mathop{\operatorname{\rm Spec}}}
\nc{\Ker}{{\mathop{\operatorname{\rm Ker}}}}
\nc{\Isom}{{\mathop{\operatorname{\rm Isom}}}}
\nc{\Hilb}{{\mathop{\operatorname{\rm Hilb}}}}
\nc{\deeq}{{\mathop{\operatorname{\rm deeq}}}}
\nc{\End}{{\mathop{\operatorname{\rm End}}}}
\nc{\Ext}{{\mathop{\operatorname{\rm Ext}}}}
\nc{\Hom}{{\mathop{\operatorname{\rm Hom}}}}
\nc{\CHom}{{\mathop{\operatorname{{\mathcal{H}}\it om}}}}
\nc{\GL}{{\mathop{\operatorname{\rm GL}}}}
\nc{\PGL}{{\mathop{\operatorname{\rm PGL}}}}
\nc{\SL}{{\mathop{\operatorname{\rm SL}}}}
\nc{\SO}{{\mathop{\operatorname{\rm SO}}}}
\nc{\Sp}{{\mathop{\operatorname{\rm Sp}}}}
\nc{\OSp}{{\mathop{\operatorname{\rm SOSp}}}}
\nc{\gr}{{\mathop{\operatorname{\rm gr}}}}
\nc{\Id}{{\mathop{\operatorname{\rm Id}}}}
\nc{\perf}{{\mathop{\operatorname{\rm perf}}}}
\nc{\defi}{{\mathop{\operatorname{\rm def}}}}
\nc{\length}{{\mathop{\operatorname{\rm length}}}}
\nc{\supp}{{\mathop{\operatorname{\rm supp}}}}
\nc{\HC}{{\mathcal H}{\mathcal C}}
\nc{\Aut}{{\mathop{\operatorname{\rm Aut}}}}

\nc{\pr}{{\operatorname{pr}}}
\nc{\Cliff}{{\mathsf{Cliff}}}
\nc{\loc}{{\operatorname{loc}}} \nc{\lc}{{\operatorname{lc}}}
\nc{\Fl}{{\mathcal{F}\ell}}
\nc{\Fib}{{\mathsf{Fib}}}
\nc{\Coh}{{\mathsf{Coh}}} \nc{\FCoh}{{\mathsf{FCoh}}}
\nc{\Perf}{{\mathsf{Perf}}}
\nc{\wtimes}{\mathbin{\widetilde\times}}

\nc{\reg}{{\text{\rm reg}}} \nc{\ren}{{\text{\rm ren}}}
\nc{\self}{{\text{\rm self}}}
\nc{\gvee}{{\mathfrak g}^{\!\scriptscriptstyle\vee}}
\nc{\tvee}{{\mathfrak t}^{\!\scriptscriptstyle\vee}}
\nc{\nvee}{{\mathfrak n}^{\!\scriptscriptstyle\vee}}
\nc{\bvee}{{\mathfrak b}^{\!\scriptscriptstyle\vee}}
       \nc{\rhovee}{\rho^{\!\scriptscriptstyle\vee}}

\nc{\cplus}{{\mathbf{C}_+}} \nc{\cminus}{{\mathbf{C}_-}}
\nc{\cthree}{{\mathbf{C}_*}} \nc{\Qbar}{{\bar{Q}}}

\newcommand\iso{\mathbin{\vphantom{j^{X^2}}\smash{\overset{\sim}{\vphantom{\rule{0pt}{0.20em}}\smash{\longrightarrow}}}}}
\nc{\Gtimes}{\vphantom{j^{X^2}}\smash{\overset{G}{\vphantom{\rule{0pt}{0.30em}}\smash{\times}}}}
\nc{\sGtimes}{\vphantom{j^{X^2}}\smash{\overset{\mathsf G}{\vphantom{\rule{0pt}{0.30em}}\smash{\times}}}}

\nc{\bOmega}{{\overline{\Omega}}}

\nc{\seq}[1]{\stackrel{#1}{\sim}}

\nc{\aff}{{\operatorname{aff}}}
\nc{\fin}{{\operatorname{fin}}}
\nc{\mir}{{\operatorname{mir}}}
\nc{\triv}{{\operatorname{triv}}}
\nc{\ext}{{\operatorname{ext}}}
\nc{\righ}{{\operatorname{right}}}
\nc{\lef}{{\operatorname{left}}}
\nc{\forg}{{\operatorname{forg}}}
\nc{\fid}{{\operatorname{fd}}}
\nc{\odd}{{\operatorname{odd}}}
\nc{\even}{{\operatorname{even}}}
\nc{\modu}{{\operatorname{-mod}}}
\nc{\Gr}{{\operatorname{Gr}}}
\nc{\FT}{{\operatorname{FT}}}
\nc{\Mat}{{\operatorname{Mat}}}
\nc{\MSt}{{\operatorname{MSt}}}
\nc{\sph}{{\operatorname{sph}}}
\nc{\GR}{{\mathbf{Gr}}}
\nc{\Perv}{{\operatorname{Perv}}}
\nc{\Rep}{{\operatorname{Rep}}}
\nc{\Ind}{{\operatorname{Ind}}}
\nc{\ind}{{\operatorname{ind}}}
\nc{\IC}{{\operatorname{IC}}}
\nc{\Bun}{{\operatorname{Bun}}}
\nc{\Proj}{{\operatorname{Proj}}}
\nc{\Stab}{{\operatorname{Stab}}}
\nc{\pt}{{\operatorname{pt}}}
\nc{\bfmu}{{\boldsymbol{\mu}}}
\nc{\bfomega}{{\boldsymbol{\omega}}}
\nc{\calM}{\mathcal M}
\nc{\calA}{\mathcal A}
\nc{\calO}{\mathcal O}
\nc{\CC}{\mathcal C}
\nc{\calN}{\mathcal N}
\nc{\grg}{\mathfrak g}
\nc{\dslash}{/\!\!/}
\nc{\tslash}{/\!\!/\!\!/}
\nc\grt{\mathfrak t}
\nc\bfM{\mathbf M}
\nc\bfN{\mathbf N}
\nc\Sig{\Sigma}
\nc\ZZ{\mathbb{Z}}
\nc\calC{\mathcal C}
\nc\calF{\mathcal F}
\nc\calX{\mathcal X}
\nc\calY{\mathcal Y}
\nc\QCoh{\operatorname{QCoh}}
\nc\IndCoh{\operatorname{IndCoh}}
\nc\Maps{\operatorname{Maps}}
\nc\Dmod{\operatorname{D-mod}}
\newcommand\Hecke{\operatorname{Hecke}}
\nc{\calD}{\mathcal D}
\nc\bfO{\mathbf O}

\nc\GG{\mathbb G}
\nc\calK{\mathcal K}
\nc{\calG}{\mathcal G}
\nc\RHom{\operatorname{RHom}}
\nc\Res{\operatorname{Res}}
\nc\Av{\operatorname{Av}}
\nc{\RH}{{\operatorname{RH}}}
\nc{\RT}{{\operatorname{RT}}}
\nc{\DR}{{\operatorname{DR}}}
\nc{\Ome}{\Omega}
\nc\Lam{\Lambda}

\nc{\bset}{{{\mathbf s}{\mathrm{et}}}}
\nc{\bSet}{{{\mathbf S}{\mathrm{et}}}}
\nc{\CSet}{{{\mathcal S}{\mathrm{et}}}}
\nc{\SET}{{\mathrm{SET}}}
\nc{\Vect}{{\mathrm{Vect}}}
\nc{\bvect}{{{\mathbf v}{\mathrm{ect}}}}
\nc{\CVect}{{{\mathcal V}{\mathrm{ect}}}}
\nc{\VECT}{{\mathrm{VECT}}}
\nc{\CRep}{{{\mathcal R}{\mathrm{ep}}}}

\nc\grs{\mathfrak s}
\nc{\tilX}{\widetilde X}
\nc\calB{\mathcal B}
\nc\calS{\mathcal S}
\nc\calT{\mathcal T}
\nc\calZ{\mathcal Z}
\nc\LS{\operatorname{LocSys}}
\nc\bfL{\on{\mathbf L}}

\newcommand*\circled[1]%{\tikz[baseline=(char.base)]{
{*{#1}}
\makeatletter
\newcommand{\raisemath}[1]{\mathpalette{\raisem@th{#1}}}
\newcommand{\raisem@th}[3]{\raisebox{#1}{$#2#3$}}
\nc{\binlim}[2][]{\def\@tempa{#1}\@ifnextchar^{\@binlim{#2}}{\@binlim{#2}^{}}}
\def\@binlim#1^#2{\mathbin{\@ifempty{#2}{\mathop{#1}}{\mathop{#1}\@xp\displaylimits\@tempa^{#2}}}}

\makeatother

\nc\cX{{\mathcal X}}

\nc\Gm{{\mathbb G_m}}
\nc{\isoto}%{\mathbin{\widetilde\to}} %
    {\stackrel{\displaystyle\raisemath{-.2ex}{\sim}}\to}
\nc\cD\CalD
\nc\setm\setminus
\nc\cE\CE
\renc\Hecke{\mathit{\CH\kern-.2ex ecke}}
\nc\bsl\backslash
\nc\Fq{\mathbb F_q}
\nc\x\times
\nc\bGO{{\bG_\bO}}
\nc\bla\blambda
\nc\blt\bullet
\nc\opp{{\on{op}}}
\nc\srel\stackrel
\nc\oast\circledast
\nc\tbx{\binlim{\widetilde\boxtimes{}}}
\nc\into\hookrightarrow
\nc\otto\leftrightarrow
\renc\setminus\smallsetminus
\nc\phitau{\varphi\tau}
\newenvironment{i-ii-iii}{%
\begin{enumerate}
}%
{\end{enumerate}}
\nc\ceil[1]{\lceil#1\rceil}  \nc\floor[1]{\lfloor#1\rfloor}
\nc\Lie{\on{Lie}}
\nc\sS{{\mathsf S}}
\nc\vvv{\ensuremath{\red\surd}}
\nc\la\lambda
\def\arxiv#1{\href{http://arxiv.org/abs/#1}{\tt arXiv:#1}} 
\nc\kap{\kappa}
\nc\gra{\mathfrak a}
\nc\gl{\mathfrak{gl}}
\nc\sTr{\operatorname{sTr}}
\nc\hatG{\widehat{G}}
\nc\calL{\mathcal L}
\nc\Whit{\operatorname{Whit}}
\nc\KL{\operatorname{KL}}
\tikzset{GL/.style={draw,%fill=blue,
    minimum size=12pt,inner sep=0pt, outer sep=0pt}}
\tikzset{SO/.style={draw,%fill=red,
    minimum size=12pt,inner sep=0pt, outer sep=0pt,font={$+$}}}
\tikzset{Sp/.style={draw,%fill=blue,
    minimum size=12pt,inner sep=0pt, outer sep=0pt,font={$-$}}}
\tikzset{pm/.style={draw,%fill=blue,
    minimum size=12pt,inner sep=0pt, outer sep=0pt,font={$\pm$}}}
\tikzset{mp/.style={draw,%fill=blue,
    minimum size=12pt,inner sep=0pt, outer sep=0pt,font={$\mp$}}}

\makeatletter
\renewcommand{\subsection}{\@startsection{subsection}{2}{0pt}{-3ex
plus -1ex minus -0.2ex}{-2mm plus -0pt minus
    -2pt}{\normalfont\bfseries}} \makeatother

\nc{\svee}{{\!\scriptscriptstyle\vee}}

\numberwithin{equation}{subsection}
\allowdisplaybreaks
\nc\mto{\mapsto }
\nc\en{\enspace }
\nc\DD{\mathbb{D}}

\begin{document}

\author[M.Finkelberg]{Michael Finkelberg}
\address{Einstein Institute of Mathematics, The Hebrew University of Jerusalem,
  Edmond J. Safra Campus, Giv’at Ram, Jerusalem, 91904, Israel;
\newline  National Research University Higher School of Economics
%  Russian Federation, Department of Mathematics, 6 Usacheva St, 119048 Moscow;
}
\email{fnklberg@gmail.com}

\author[D.Kazhdan]{David Kazhdan}
\address{Einstein Institute of Mathematics, The Hebrew University of Jerusalem,
  Edmond J. Safra Campus, Giv’at Ram, Jerusalem, 91904, Israel}
\email{kazhdan@mail.huji.ac.il}

\title[$\Gamma$-function over a two-dimensional local field]
  {An analogue of $\Gamma$-function for irreducible cuspidal representations
  of the group $\GL(2)$ over a two-dimensional local field}

\begin{abstract}
Let $\bF$ be a non-archimedian local field of odd residue characteristic,
$\CO=\bF[\![t]\!]\subset\CK=\bF(\!(t)\!)$. We consider the subgroup
$\wh\GL(2,\CO)=\CK^\times\cdot\GL(2,\CO)\subset\GL(2,\CK)$, and the normalizer
$\GL'(2,\CO)\subset\GL(2,\CK)$ of an Iwahori subgroup $I\subset\GL(2,\CK)$.
%generated by the upper Iwahori subgroup and $\begin{pmatrix}0&1\\ -t&0\end{pmatrix}$.
  We calculate the $\Gamma$-functions for irreducible special representations
  of $\wh\GL(2,\CO)$ and $\GL'(2,\CO)$.
\end{abstract}  

\maketitle
  
\section{Introduction}
\subsection{$\Gamma$-functions for local fields}
The notion of $\Gamma$-functions for representations of a general linear
group over a nonarchimedian local field $\bF$ goes back to~\cite{GK}, cf.~\cite{JPS} and~\cite{O}.
Under the Langlands correspondence with Galois representations, the $\Gamma$-function of a cuspidal representation goes to the $\varepsilon$-factor of the
corresponding Galois representation.

The theory of local $\Gamma$-functions associates numerical invariants to
representations of general linear groups over non-Archimedean local fields.
For an irreducible cuspidal representation of $\mathrm{PGL}(2)$, these
invariants admit a direct representation-theoretic description: in the
Kirillov model, the nontrivial Weyl element takes a multiplicative character
to its inverse, multiplied by a scalar. As the character varies, these
scalars form the $\Gamma$-function of the representation. This description
provides a starting point for defining analogous invariants over
two-dimensional local fields.

Let $\CK=\bF(\!(t)\!)\supset\CO=\bF[\![t]\!]$, and assume that the residue characteristic of $\bF$
is odd. 

Braverman--Kazhdan~\cite{BK} constructed analogues of irreducible cuspidal
representations over $\mathcal K$, using special representations of certain
bounded subgroups of $\mathrm{PGL}(2,\mathcal K)$. In this paper we define
and calculate $\Gamma$-functions for the special representations considered
in that construction. We treat both $\mathrm{PGL}(2,\mathcal O)$ and the
normalizer $G'(\mathcal O)$ of an Iwahori subgroup, and explain the
modifications required for $\mathrm{GL}(2)$.

The defining property of a special representation is the irreducibility
of its restriction to the upper triangular mirabolic subgroup. At level
$n$, this restriction has a realization in the Schwartz space
\[
\mathcal S(\mathcal O_n^\times),
\qquad
\mathcal O_n=\mathcal O/t^{n+1}\mathcal O.
\]
A multiplicative character $\omega$ of $\mathcal O_n^\times$ can be viewed
as a generalized vector in the corresponding dual space. For a special
representation $\pi_\theta$, we define the scalar
$\Gamma(\pi_\theta,\omega;\psi_n)$ by
\[
\pi_\theta^\wedge(w)\omega
=\Gamma(\pi_\theta,\omega;\psi_n)\,\omega^{-1},
\]
where
\[
w=\begin{pmatrix}0&1\\-1&0\end{pmatrix}
\quad\text{or}\quad
w=\begin{pmatrix}0&1\\-t&0\end{pmatrix},
\]
according to the bounded subgroup under consideration. Here $\psi_n$ is
obtained from a fixed nontrivial additive character of $\mathbf F$ by
taking the coefficient of $t^n$.

Our main results give explicit formulas for these scalars in terms of the
torus character $\theta$, the twisting character $\omega$, and their
restrictions to congruence subgroups. The calculations distinguish the
unramified and ramified quadratic extensions of $\mathcal K$ that
parametrize the representations. In the unramified case, odd and even
depths require different arguments. At odd depth, the calculation reduces
to character evaluations. At even depth, a Heisenberg representation and
its Weil operators enter the construction, and the resulting formula
contains a quadratic Weil index and a factor determined by the chosen
splitting over the torus. In the ramified case, the formulas are expressed
directly through character evaluations.

The proofs use explicit intertwining operators between models of the
mirabolic representation. These operators reduce the calculation of the
Weyl action to Fourier inversion and, in the even-depth case, a
one-dimensional quadratic oscillatory integral over $\mathbf F$. Keeping
the normalizations explicit also describes the dependence on the additive
character and on the splitting used to parametrize the representations.

We begin by recalling the corresponding construction over a
one-dimensional local field. We then carry out the two-dimensional
calculations, first for the unramified and then for the ramified case.
Finally, we give analogues for the classical case of $\GL(2,\bF)$ and extend the construction
to $\mathrm{GL}(2)$, where the central character $\mu$ replaces the target
character $\omega^{-1}$ by $\mu\omega^{-1}$.

\subsection{Acknowledgments}
We are grateful to N.~Gurevich for the useful discussions, and to ChatGPT~6 for the numerous corrections, references, and for
the proof of~Lemmas~\ref{RRA},\ref{Gerardin}.
During the preparation of this paper, D.K.\ was partially supported by ERC grant no.~101142781, while
the research of M.F.\ was supported by the Israel Science
Foundation (grant No.\ 994/24).

\section{Reminder about representations of $\PGL(2,\bF)$}
\label{field F}

\subsection{Setup}
\label{setup}
Let $\bF$ be a local field of odd residue characteristic. Let $\bO\subset\bF$ be the ring of integers, and let
$\bk=\BF_q$ be the residue field. Let $\varpi\in\bO$ be a uniformizer.
We set $G=\PGL(2)\supset P=\left\{\begin{pmatrix}a&b\\ 0&1\end{pmatrix}\right\}$. Recall that an irreducible representation of $G(\bF)$
is cuspidal iff its restriction to $P(\bF)$ is isomorphic to a unique infinite-dimensional irreducible representation of $P(\bF)$.
The latter one is the compact induction $\ind_{\bF}^{P(\bF)}(\psi)\simeq\CS(\bF^\times)$. Here $\bF\subset P(\bF)$ is the group of strictly
upper triangular matrices, $\CS(\bF^\times)$ is the Schwartz space of compactly supported locally constant functions, and
$\psi$ is a nontrivial additive character $\bF\to\BC^\times$.

Any maximal compact subgroup of $G(\bF)$ is conjugate to either $G(\bO)$ or $G'(\bO)$: the subgroup generated by
$\begin{pmatrix}0&1\\ -\varpi&0\end{pmatrix}$ and the Iwahori subgroup
$I=\left\{\begin{pmatrix}a&b\\ c&d\end{pmatrix} : a,d\in\bO^\times,\ b\in\bO,\ c\in \varpi\bO\right\}\subset G'(\bO)$.
Any irreducible cuspidal representation $\Pi$ of $G(\bF)$ is compactly induced from a {\em special} finite dimensional representation $\pi$ of
$G(\bO)$ or $G'(\bO)$, see e.g.~\cite[Chapter 4]{BH}. Recall that $\pi$ is called special if its restriction to $P(\bO)$ is an irreducible
representation of $P(\bO)$ of dimension~$>1$. The latter representation is realized in the space of functions $\BC[\bO^\times_n]$ for some
$n\in\BN$, where $\bO^\times_n=\bO^\times/(1+\varpi^{n+1}\bO)$, cf.~\cite[\S4]{BK}. It factors through $P(\bO)\twoheadrightarrow P(\bO)_n=P(\bO)/P(\bO)(n+1)$,
where $P(\bO)(n+1)$ is the intersection of $P(\bO)$ with the $n+1$-th congruence subgroup of $G(\bO)$. The corresponding representation
of $P(\bO)_n$ is induced from an additive character $\bar\psi_n$ of the strictly upper triangular subgroup $\bO_n\subset P(\bO)_n$
such that $\bar\psi_n$ is nontrivial on the lowest level subgroup $\bk\simeq\bO_n(n)\subset\bO_n$.

The special representations are indexed by certain characters of $T^\bE$,
where $\bE$ is a quadratic field extension of $\bF$, and $T^\bE=\bE^\times/\bF^\times$.

\subsection{$\Gamma$-functions}
\label{Gamma classical}
For the sake of definiteness, consider a special representation $\pi$ of $G'(\bO)$ of depth $n+\frac12$.
(The discussion below has a direct analogue for $G(\bO)$.) Then $\pi$ corresponds to a character ${\theta}\colon T^\bE\to\BC^\times$
for a ramified quadratic extension $\bE$ of $\bF$, $\pi=\pi_{\theta}$. We choose a uniformizer $\varpi\in\bO$ in such a way that
$\bE=\bF(\sqrt{-\varpi})$. Given a character $\bar\omega\colon\bO^\times_n\to\BC^\times$,
we consider it as a function $\bar\omega\in\BC[\bO^\times_n]$. We view $\BC[\bO^\times_n]$ as the underlying space of the special irreducible
representation $\pi_{\theta}$ of $G'(\bO)$ via the identifications of~\S\ref{setup}, and apply to $\bar\omega\in\BC[\bO^\times_n]$ the operator
$\pi_{\theta}\begin{pmatrix}0&1\\ -\varpi&0\end{pmatrix}$. Then we get $\bar\omega^{-1}$ up to a constant multiple
$\Gamma(\pi_{\theta},\bar\omega;\bar\psi_n)$.

On the other hand, we extend $\bar\omega\colon\bO^\times\twoheadrightarrow\bO^\times_n\to\BC^\times$ to a multiplicative character
$\omega\colon\bF^\times\to\BC^\times$ such that $\omega(\varpi)=1$. We also extend the additive character
$\bar\psi_n\colon\bO\twoheadrightarrow\bO_n\to\BC^\times$ to an additive character $\psi\colon\bF\to\BC^\times$.

The compact induction $\on{ind}_{G'(\bO)}^{G(\bF)}(\pi_{\theta})$ is an irreducible representation $\Pi$ of $G(\bF)=\PGL(2,\bF)$,
and $\pi$ is contained in $\Pi$ as $\Pi^{G'(\bO)(n+1)}$.
Again, the restriction of $\Pi$ to $P(\bF)$ is irreducible; it is induced from the additive character $\psi$.
It is realized in the Schwartz space $\CS(\bF^\times)$. We can view $\omega\in\BC(\bF^\times)$ as an element of
the conjugate dual Schwartz space $\CS'(\bF^\times)$.
Using the invariant unitary structure on $\CS(\bF^\times)$ we embed $\CS(\bF^\times)$ into $\CS'(\bF^\times)$ and
extend the action of $G(\bF)$ on $\CS(\bF^\times)$ to the action on $\CS'(\bF^\times)$.
Then it is known that $\Pi\begin{pmatrix}0&1\\ -1&0\end{pmatrix}\omega\in\CS'(\bF^\times)$
is proportional to $\omega^{-1}$, that is
$\Pi\begin{pmatrix}0&1\\ -1&0\end{pmatrix}\omega=\Gamma(\Pi,\omega;\psi)\omega^{-1}$ for a scalar
$\Gamma(\Pi,\omega;\psi)\in\BC^\times$.

Recall that the representation $\Pi$ corresponds by Langlands to the 2-dimensional
representation $V$ of $\on{Gal}(\bF)$, and $\omega$ corresponds by the local class field theory to the same named character of $\on{Gal}(\bF)$.
Under the Langlands correspondence~\cite[Chapter 8]{BH},
$\Gamma(\Pi,\omega;\psi)=\varepsilon(\frac12,V\otimes\omega^{-1};\psi)$.

\begin{lem}
  \label{GamGam}
  $\Gamma(\pi_{\theta},\bar\omega;\bar\psi_n)=\Gamma(\Pi,\omega;\psi)$.
\end{lem}

\begin{proof}
For $k\in\BZ$, set $f_k(x)=\omega(x)\delta_{v(x),k}$ (Kronecker delta function of the valuation of $x\in\bF^\times$),
and $f_k^-(x)=\omega(x)^{-1}\delta_{v(x),k}$
(functions in $\CS(\bF^\times)$). Then $\Pi\begin{pmatrix}0&1\\ -\varpi&0\end{pmatrix}f_0=\Gamma(\pi_{\theta},\bar\omega;\bar\psi_n)f_0^-$.
Since $\begin{pmatrix}0&1\\ -\varpi&0\end{pmatrix}=\begin{pmatrix}\varpi^{-1}&0\\ 0&1\end{pmatrix}\begin{pmatrix}0&1\\ -1&0\end{pmatrix}$,
we get $\Pi\begin{pmatrix}0&1\\ -1&0\end{pmatrix}f_0(x)=\Gamma(\pi_{\theta},\bar\omega;\bar\psi_n)f_0^-(\varpi x)$, and
$\Pi\begin{pmatrix}0&1\\ -1&0\end{pmatrix}f_k(x)=\Gamma(\pi_{\theta},\bar\omega;\bar\psi_n)f_{-k-1}^-(x)$.
Given $z\in S^1$, we set $\chi_z(x)=z^{v(x)}$. By Fourier expansion, $\omega\chi_z=\sum_{k\in\BZ}z^kf_k$.
Hence $\Gamma(\Pi,\omega\chi_z;\psi)=z^{-1}\Gamma(\pi_{\theta},\bar\omega;\bar\psi_n)$. It remains to take $z=1$.
\end{proof}

\section{2-dimensional case}

\subsection{Representations of $P(\CO)$}
\label{Rep PO}
Recall that $\bF$ is a local field of odd residue characteristic, and $\psi\colon\bF\to\BC^\times$ is a nontrivial additive character.
We set $\CK=\bF(\!(t)\!)\supset\CO=\bF[\![t]\!]$. For $k\in\BZ$ we consider the additive character $\psi_k\colon\CK\to\BC^\times$ equal
to the composition of $\psi$ with taking the $k$-th coefficient of a Laurent series.

We follow the notation of~\cite{BK}. In particular, we set $G(\CO)=\PGL(2,\CO)$, and for $n\in\BN$, $G(\CO)(n)\subset G(\CO)$ is the
$n$-th congruence subgroup. We set $G(\CO)_n=G(\CO)/G(\CO)(n+1)$. Furthermore, $P(\CO)\subset G(\CO)$ is the upper triangular mirabolic
subgroup, and $P(\CO)(n)=P(\CO)\cap G(\CO)(n)$, $P(\CO)_n=P(\CO)/P(\CO)(n+1)$. We have embeddings
$\CO=\left\{\begin{pmatrix}1&z\\ 0&1\end{pmatrix},\ z\in\CO\right\}\subset P(\CO)$ and
$\CO^\times=\left\{\begin{pmatrix}x&0\\ 0&1\end{pmatrix},\ x\in\CO^\times\right\}\subset P(\CO)$, so that $P(\CO)=\CO^\times\ltimes\CO$,
and $P(\CO)_n=\CO^\times_n\ltimes\CO_n$.

Recall from~\cite[Definition 4.1.(4)]{BK} the irreducible representation $\sigma_n=\ind_{\CO_n}^{P(\CO)_n}(\psi_n)$ of $P(\CO)_n$
(compact induction) in the Schwartz space $\CS(\CO^\times_n)$. Let
%\begin{equation}
%  \label{Psi}
 $ \Psi\colon\CO_n\to\BC^\times$
%\end{equation}
be an additive character such that
$\Psi|_{\CO_n(n)}\not\equiv1$. Then there is $c\in\CO^\times_n$ such that $\Psi(y)=\psi_n(cy)$ for any $y\in\CO_n$.
The representation $\ind_{\CO_n}^{P(\CO)_n}(\Psi)$ is realized in the space of functions on $P(\CO)_n$ with compact support modulo $\CO_n$
such that $\varphi\left(\begin{pmatrix}1&y\\ 0&1\end{pmatrix}\begin{pmatrix}x&z\\ 0&1\end{pmatrix}\right)
=\Psi(y)\varphi\begin{pmatrix}x&z\\ 0&1\end{pmatrix}$.

\begin{lem}
  \label{trivial}
  An intertwining operator $S\colon\sigma_n=\ind_{\CO_n}^{P(\CO)_n}(\psi_n)\to\ind_{\CO_n}^{P(\CO)_n}(\Psi)$ is given
  by $S\varphi\begin{pmatrix}x&z\\ 0&1\end{pmatrix}
  =\varphi\left(\begin{pmatrix}c&0\\ 0&1\end{pmatrix}\begin{pmatrix}x&z\\ 0&1\end{pmatrix}\right).$ \hfill $\Box$
\end{lem}

We will need the following subgroups of $P(\CO)_n$. If $n=2m-1$ is odd, we have the abelian subgroup $P(\CO)_{2m-1}(m)\subset P(\CO)_{2m-1}$
(the image of $P(\CO)(m)$ in $P(\CO)_{2m-1}$). If $n=2m$ is even, we have the abelian subgroup
$P'(\CO)_{2m}(m)=\left\{\begin{pmatrix}\CO_{2m}^\times(m+1)&\CO_{2m}(m)\\ 0&1\end{pmatrix}\right\}$. In the odd case, let us assume
that an additive character $\Psi\colon\CO_{2m-1}\to\BC^\times$ factors through the projection $\CO_{2m-1}\to\CO_{2m-1}(m)$,
$b_0+b_1t+\ldots+b_{2m-1}t^{2m-1}\mapsto b_mt^m+\ldots+b_{2m-1}t^{2m-1}$. In the even case, let us assume that $\Psi\colon\CO_{2m}\to\BC^\times$
factors through the projection $\CO_{2m}\to\CO_{2m}(m)$, $b_0+b_1t+\ldots+b_{2m}t^{2m}\mapsto b_mt^m+\ldots+b_{2m}t^{2m}$.
We also consider the character $\Xi\colon P(\CO)_{2m-1}(m)\to\BC^\times$ (resp.\ $\Xi\colon P'(\CO)_{2m}(m)\to\BC^\times$) defined by
$\Xi\begin{pmatrix}x&z\\ 0&1\end{pmatrix}=\Psi(z)$.

\begin{lem}
  \label{intertwiners}
  \textup{(1)} An isomorphism $T\colon\ind_{\CO_{2m-1}}^{P(\CO)_{2m-1}}(\Psi)\iso\ind_{P(\CO)_{2m-1}(m)}^{P(\CO)_{2m-1}}(\Xi)$ is given by
  $Tf\begin{pmatrix}x&z\\ 0&1\end{pmatrix}
  =\int_{ a\in\CO_{m-1}}f\left(\begin{pmatrix}1+t^m a&0\\ 0&1\end{pmatrix}\begin{pmatrix}x&z\\ 0&1\end{pmatrix}\right)d a.$

  \textup{(2)} The inverse isomorphism is given by\\
  $T^{-1}f\begin{pmatrix}x&z\\ 0&1\end{pmatrix}
  =\int_{ b\in\CO_{m-1}}f\left(\begin{pmatrix}1& b\\ 0&1\end{pmatrix}\begin{pmatrix}x&z\\ 0&1\end{pmatrix}\right)d b.$

  \textup{(3)} An isomorphism $T\colon\ind_{\CO_{2m}}^{P(\CO)_{2m}}(\Psi)\iso\ind_{P'(\CO)_{2m}(m)}^{P(\CO)_{2m}}(\Xi)$ is given by
  $Tf\begin{pmatrix}x&z\\ 0&1\end{pmatrix}
  =\int_{ a\in\CO_{m-1}}f\left(\begin{pmatrix}1+t^{m+1} a&0\\ 0&1\end{pmatrix}\begin{pmatrix}x&z\\ 0&1\end{pmatrix}\right)d a.$

  \textup{(4)} The inverse isomorphism is given by\\
  $T^{-1}f\begin{pmatrix}x&z\\ 0&1\end{pmatrix}
  =\int_{ b\in\CO_{m-1}}f\left(\begin{pmatrix}1& b\\ 0&1\end{pmatrix}\begin{pmatrix}x&z\\ 0&1\end{pmatrix}\right)d b.$ 

Here $\int_{\CO_{m-1}}$ is taken with respect to a self-dual Haar measure on $\CO_{m-1}=\bF^m$ with respect to the pairing
$\langle a, b\rangle=\Psi(t^m\cdot a\cdot b)=\psi_n(c\cdot t^m\cdot a\cdot b)$ in case $n=2m-1$
(resp.\ $\langle a, b\rangle=\Psi(t^{m+1}\cdot a\cdot b)=\psi_n(c\cdot t^{m+1}\cdot a\cdot b)$ in case $n=2m$)
(recall that $c$ was introduced right before~Lemma~\ref{trivial}).   \hfill $\Box$
\end{lem}

Finally, we have $\ind_{P'(\CO)_{2m}(m)}^{P(\CO)_{2m}}(\Xi)\cong\ind_{P(\CO)_{2m}(m)}^{P(\CO)_{2m}}(\ind_{P'(\CO)_{2m}(m)}^{P(\CO)_{2m}(m)}(\Xi))$.
Here $\ind_{P'(\CO)_{2m}(m)}^{P(\CO)_{2m}(m)}(\Xi)=:\CW_\Xi\cong\CS(\bF)$ is realized in the space of locally constant functions $f$ on
$P(\CO)_{2m}(m)$ compactly supported modulo $P'(\CO)_{2m}(m)$ such that $f\left(\begin{pmatrix}a&b\\ 0&1\end{pmatrix}
\begin{pmatrix}x&z\\ 0&1\end{pmatrix}\right)=\Xi\begin{pmatrix}a&b\\ 0&1\end{pmatrix}f\begin{pmatrix}x&z\\ 0&1\end{pmatrix}$
for any $\begin{pmatrix}a&b\\ 0&1\end{pmatrix}\in P'(\CO)_{2m}(m)$. The identification of $\CW_\Xi$ with $\CS(\bF)$ is given by
restriction of functions to the line $\left\{\begin{pmatrix}1+t^m\xi&0\\ 0&1\end{pmatrix},\ \xi\in\bF\right\}$.    

\begin{lem}
  \label{banal}
\textup{(1)}  An isomorphism $S\colon\ind_{P'(\CO)_{2m}(m)}^{P(\CO)_{2m}}(\Xi)\iso\ind_{P(\CO)_{2m}(m)}^{P(\CO)_{2m}}(\CW_\Xi)$ is given by
$S\varphi\begin{pmatrix}x&z\\ 0&1\end{pmatrix}
=\varphi\left(\begin{pmatrix}1+t^m\xi&0\\ 0&1\end{pmatrix}\begin{pmatrix}x&z\\ 0&1\end{pmatrix}\right)$ (i.e.\ the value of
$S\varphi$ at $\begin{pmatrix}x&z\\ 0&1\end{pmatrix}$ is the function of $\xi\in\bF$ in the RHS).

\textup{(2)} The inverse isomorphism $S^{-1}\colon\ind_{P(\CO)_{2m}(m)}^{P(\CO)_{2m}}(\CW_\Xi)\iso\ind_{P'(\CO)_{2m}(m)}^{P(\CO)_{2m}}(\Xi)$ is given by
$S^{-1}\Phi\begin{pmatrix}x&z\\ 0&1\end{pmatrix}=\Phi\begin{pmatrix}x&z\\ 0&1\end{pmatrix}(0)$ (in the RHS we evaluate the
$\CS(\bF)$-valued function $\Phi$ of $\begin{pmatrix}x&z\\ 0&1\end{pmatrix}$ at $0\in\bF$). 
\hfill $\Box$
\end{lem}

\subsection{Representations of $G(\CO)$, odd depth}
\label{odd depth}
Let $\CE$ be an unramified
quadratic extension of $\CK$, namely $\CE=\bE(\!(t)\!)$, where $\bE=\bF(\sqrt{\epsilon})$ is a quadratic extension of $\bF$,
and $\epsilon\in\bF^\times\setminus(\bF^\times)^2$. We fix an embedding
$T^\CE=\CE^\times/\CK^\times\hookrightarrow\PGL(2,\CK)=G(\CK)$ with the image consisting of the classes of the
elements $\begin{pmatrix}a&b\\ b\epsilon&a\end{pmatrix}$. Note that the image actually lies in the subgroup
$G(\CO)\subset G(\CK)$. 
We have the $\on{Ad}_{T^\CE}$-invariant projection $p\colon\fg(\CK)\to\ft^\CE$,
$\begin{pmatrix}a&b\\ c&-a\end{pmatrix}\mapsto\begin{pmatrix}0&\frac{c\epsilon^{-1}+b}{2}\\
\frac{b\epsilon+c}{2}&0\end{pmatrix}$.
Under the natural identifications, $p$ induces the same named map from $G(\CO)_{2m-1}(m)$ to $T^\CE_{2m-1}(m)$, and
from $G(\CO)_{2m}(m)$ to $T^\CE_{2m}(m)$, and the former map is a homomorphism, while the latter is not.

Suppose first that $n=2m-1$ is odd. Given a depth $2m-1$ character $\theta\colon T^\CE_{2m-1}\to\BC^\times$,
we consider the character $\theta\circ p$ of the normal abelian subgroup $G(\CO)_{2m-1}(m)$, 
and its stabilizer in $G(\CO)_{2m-1}$ is $T^\CE_{2m-1}\cdot G(\CO)_{2m-1}(m)$.
We assume that the action of the Galois group of $\CE:\CK$ does not fix $\theta$, that is $\bar\theta\ne\theta$. Then the induced representation
$\pi_\theta=\ind_{\on{Stab}(\theta\circ p)}^{G(\CO)_{2m-1}}(\theta)$ is irreducible and special (see~\cite[\S5.5]{BK}):
its restriction to $P(\CO)_n$ remains irreducible and is naturally isomorphic to $\ind_{P(\CO)_{2m-1}(m)}^{P(\CO)_{2m-1}}(\varXi)$.
Hence this restriction is isomorphic to the representation $\sigma_{2m-1}=\ind_{\CO_{2m-1}}^{P(\CO)_{2m-1}}(\psi_{2m-1})$,
see~Lemmas~\ref{trivial},\ref{intertwiners}. Here the character $\varXi$ is defined as follows.

The abelian subgroup $P(\CO)_{2m-1}(m)\subset P(\CO)_{2m-1}$
has the character (also denoted by $\theta\circ p$) given by the following composition:
\begin{equation}
  \label{theta eps}
  P(\CO)_{2m-1}(m)\hookrightarrow G(\CO)_{2m-1}(m)\xrightarrow{p}T^{\CE}_{2m-1}(m)\xrightarrow{\theta}\BC^\times.
  \end{equation}
From the explicit formula for $p$ it is clear that $\theta\circ p$ factors through the projection to the strictly upper triangular
matrices $\CO_{2m-1}(m)\subset P(\CO)_{2m-1}(m)$: 
\begin{equation}
  \label{Psi eps}
  P(\CO)_{2m-1}(m)\twoheadrightarrow\CO_{2m-1}(m)\xrightarrow{\varXi}\BC^\times,
  \end{equation}
where $\varXi$ is defined as follows. For an element $ b=b_mt^m+\ldots+b_{2m-1}t^{2m-1}\in\CO_{2m-1}(m)$,
we have $\varXi( b)=\theta(1+\frac{\sqrt{\epsilon}}{2} b)$. 
By an abuse of notation, we will denote the composition~\eqref{Psi eps} by $\varXi$ as well.
      
We can also consider the strictly upper triangular subgroup $\CO_{2m-1}\subset P(\CO)_{2m-1}$, and its character $\varPsi$ given by the formula
\begin{equation}
  \label{Xi eps}
  \varPsi(b_0+b_1t+\ldots+b_mt^m+\ldots b_{2m-1}t^{2m-1})=\theta(1+\frac{\sqrt{\epsilon}}{2}(b_mt^m+\ldots+b_{2m-1}t^{2m-1}))
\end{equation}
(so that the characters~\eqref{Psi eps} and~\eqref{Xi eps} coincide on the intersection of $P(\CO)_{2m-1}(m)$ and $\CO_{2m-1}$).
In other words, for $ a\in\CO_{m-1}$,
\begin{equation}
  \label{ctheta eps}
  \varPsi(t^m a)=\theta(1+\frac{\sqrt{\epsilon}t^m}{2} a)=\psi_{m-1}(c_\theta\cdot a),
\end{equation}
where $\psi_{m-1}(b_0+b_1t+\ldots+b_{m-1}t^{m-1})=\psi(b_{m-1})$, and $c_\theta\in\CO_{m-1}^\times$.

\subsection{Representations of $G(\CO)$, even depth}
\label{even depth}
Now $n=2m>0$ is even. A depth $2m$ character $\theta\colon T^\CE_{2m}\to\BC^\times$ gives rise to an irreducible representation of a
certain subgroup $H\subset G(\CO)_{2m}$. Namely,
we consider the character $\theta\circ p$ of the normal abelian subgroup $G(\CO)_{2m}(m+1)\subset G(\CO)_{2m}$, its stabilizer
$\on{Stab}_\theta:=\on{Stab}_{G(\CO)_{2m}}(\theta\circ p)=T^\CE_{2m}\cdot G(\CO)_{2m}(m)$, and its quotient
$\ol{\on{Stab}}_\theta:=\on{Stab}_\theta/\on{Ker}(\theta\circ p)$ (a {\em quasi-Heisenberg} group, see~\cite[\S3.7]{BK}).
We assume again that the action of the Galois group of $\CE:\CK$ does not fix $\theta$, that is $\bar\theta\ne\theta$.
Then the induced representation
$\pi_\theta=\ind_{\on{Stab}_\theta}^{G(\CO)_{2m}}(W_\theta)$ of an irreducible representation of $\on{Stab}_\theta$ factoring through an irreducible
representation of $\ol{\on{Stab}}_\theta$, is irreducible and special (see~\cite[\S5.6]{BK}): its restriction to $P(\CO)_n$
remains irreducible and is naturally
isomorphic to $\ind_{P'(\CO)_{2m}(m)}^{P(\CO)_{2m}}(\varXi)$.  
Hence this restriction is isomorphic to the representation $\sigma_{2m}=\ind_{\CO_{2m}}^{P(\CO)_{2m}}(\psi_{2m})$,
see~Lemmas~\ref{trivial},\ref{intertwiners}. Here the character $\varXi$ is defined as follows.

Recall the abelian subgroup
$P'(\CO)_{2m}(m)=\left\{\begin{pmatrix}\CO_{2m}^\times(m+1)&\CO_{2m}(m)\\ 0&1\end{pmatrix}\right\}$. Note that the restriction of
$p\colon G(\CO)_{2m}(m)\to T^\CE_{2m}(m)$ to $P'(\CO)_{2m}(m)\subset G(\CO)_{2m}(m)$ is a homomorphism. So $P'(\CO)_{2m}(m)$ 
has the character given by the composition 
\begin{equation}
  \label{theta even eps}
  P'(\CO)_{2m}(m)\hookrightarrow G(\CO)_{2m}(m)\xrightarrow{p}T^{\CE}_{2m}(m)\xrightarrow{\theta}\BC^\times.
  \end{equation}
Again $\theta\circ p$ factors through the projection to the strictly upper triangular
matrices $\CO_{2m}(m)\subset P'(\CO)_{2m}(m)$: 
\begin{equation}
  \label{Psi even eps}
  P'(\CO)_{2m}(m)\twoheadrightarrow\CO_{2m}(m)\xrightarrow{\varXi}\BC^\times,
  \end{equation}
where $\varXi$ is defined as follows. For an element $ b=b_mt^m+\ldots+b_{2m}t^{2m}\in\CO_{2m}(m)$, we have
$\varXi( b)=\theta(1+\frac{\sqrt{\epsilon}}{2} b)$.
By an abuse of notation, we will denote the composition~\eqref{Psi even eps} by $\varXi$ as well.

We can also consider the strictly upper triangular subgroup $\CO_{2m}\subset P(\CO)_{2m}$, and its character $\varPsi$ given by the formula
\begin{equation}
  \label{Xi even eps}
  \varPsi(b_0+b_1t+\ldots+b_mt^m+\ldots+b_{2m}t^{2m})=\theta(1+\frac{\sqrt{\epsilon}}{2}(b_mt^m+\ldots+b_{2m}t^{2m}))
\end{equation}
(so that the characters~\eqref{Psi even eps} and~\eqref{Xi even eps} coincide on the intersection of $P'(\CO)_{2m}(m)$ and $\CO_{2m}$).
Similarly to~\eqref{ctheta eps}, we have $\varPsi(t^{m+1} a)=\theta(1+\frac{\sqrt{\epsilon}t^{m+1}}{2} a)=\psi_{m-1}(c_\theta\cdot a)$.

Finally, recall from~\cite[the paragraph above Proposition 5.31]{BK} that the construction of the quasi-Heisenberg representation $\CW_\theta$ depends on
a choice of lift $s\colon T^\CE_{2m}\twoheadrightarrow T^\CE_0=T^\bE\to\on{Mp}(2,\bF)$ to the metaplectic cover, of the homomorphism
$T^\bE\to\Sp(2,\bF)$ arising from the natural action of $T^\bE$ on $\on{Ker}(p\colon \fg(\CO)_m(m)\to\ft^\CE_m(m))$.

\begin{lem}
  \label{TE}
  In the basis $\left\{\begin{pmatrix}\frac{t^m}{2}&0\\ 0&\frac{-t^m}{2}\end{pmatrix},\ \begin{pmatrix}0&\frac{t^m}{2}\\ \frac{-\epsilon t^m}{2}&0\end{pmatrix}\right\}$ of
  $\on{Ker}(p\colon \fg(\CO)_m(m)\to\ft^\CE_m(m))$, the adjoint action of an element $\begin{pmatrix}\alpha&\beta\\ \beta\epsilon&\alpha\end{pmatrix}\in T^\bE$
  is given by the matrix $\begin{pmatrix}\frac{\alpha^2+\beta^2\epsilon}{\alpha^2-\beta^2\epsilon}&\frac{-2\alpha\beta\epsilon}{\alpha^2-\beta^2\epsilon}\\ \frac{-2\alpha\beta}{\alpha^2-\beta^2\epsilon}&
  \frac{\alpha^2+\beta^2\epsilon}{\alpha^2-\beta^2\epsilon}\end{pmatrix}\in\Sp(2,\bF)$. \hfill $\Box$
\end{lem}

Fix $r\in\bF^\times$. Consider the Weil representation $\CW_{\psi_r}$ of $\on{Mp}(2,\bF)$ in the Schwartz space $\CS(\bF)$ arising
from the Heisenberg action in $\CS(\bF)$ with central character $\psi_r(x)=\psi(rx)$. More precisely, the Heisenberg group $\CH$
is the central extension of $\bF^2$ with the help of $\bF$, with multiplication
$(a,b;e)\cdot(a',b';e')=(a+a',b+b';e+e'+\frac12(ab'-a'b))$. Its irreducible representation $\rho$ with central character $\psi_r$
is realized in $\CS(\bF) : \rho(a,b,0)\phi(\xi)=\psi(r(b\xi+\frac12 ab))\phi(\xi+a)$. The corresponding Weil operators for
the projective action of $\Sp(2,\bF)$ in $\CS(\bF)$ satisfy the relation $A(X)\rho(a,b)A(X)^{-1}=\rho(X(a,b))$, where we view $(a,b)$ as
a column vector. In other words, for $X=\begin{pmatrix}x_{11}&x_{12}\\ x_{21}&x_{22}\end{pmatrix}$, $X(a,b)=(x_{11}a+x_{12}b,x_{21}a+x_{22}b)$.
Now consider the torus $\bE^1=\left\{\begin{pmatrix}u&\epsilon v\\ v&u\end{pmatrix} : u^2-\epsilon v^2=1\right\}\subset\Sp(2,\bF)$.
Then $A\begin{pmatrix}u&\epsilon v\\ v&u\end{pmatrix}\rho(a,b)A\begin{pmatrix}u&\epsilon v\\ v&u\end{pmatrix}^{-1}=
\rho(ua+\epsilon vb,va+ub)$.
The restriction of the metaplectic 2-fold cover $\on{Mp}(2,\bF)\to\Sp(2,\bF)$ to the compact subgroup $\bE^1\subset\Sp(2,\bF)$ splits.
In fact, there are two splittings, whose ratio is a quadratic character of $\bE^1$.
The following lemma records an explicit formula for a particular splitting. To formulate it we need a couple more notations.
For $a\in\bF^\times$, we denote by $\gamma\!\left(y\mapsto\psi\left(\frac{ay^2}{2}\right)\right)$ the absolute Weil index of the quadratic character
$\bF\ni y\mapsto\psi\left(\frac{ay^2}{2}\right)$. For $u+v\sqrt{\epsilon}\in\bE^1$, we denote by $\zeta(u+v\sqrt{\epsilon})=\pm1$ the
factor defined in~\cite[Theorem 5.7]{A} (in fact, there are two possibilities for $\zeta$ depending on additional choices; and we always
choose $\alpha=1$ in {\em loc.cit.}).
For $a,b\in\bF^\times$, $(a,b)_\bF$ stands for the Hilbert symbol.

\begin{lem}
  \label{RRA}
  A homomorphic lift of $\bE^1=\left\{\begin{pmatrix}u&\epsilon v\\ v&u\end{pmatrix} : u^2-\epsilon v^2=1\right\}$ to the 2-fold cover
  $\Sp(2,\bF)\leftarrow\on{Mp}(2,\bF)\subset\on{Aut}(\CS(\bF))$ is given for $v\ne0$ by the following automorphism $A_{u,v}$ of the
  Schwartz space $\CS(\bF)$:
  \begin{multline*}A_{u,v}\phi(\xi)=
  \zeta(u+v\sqrt{\epsilon})\gamma\!\left(y\mapsto\psi\left(-\frac{rvy^2}{2\epsilon}\right)\right)^{-1}|\epsilon v|_\bF^{-1/2}|r|_\bF^{1/2}\\
  \times\int_{\eta\in\bF}\psi\left(-\frac{r}{2\epsilon v}(u\xi^2+u\eta^2-2\xi\eta)\right)\phi(\eta)d\eta,
  \end{multline*}
  where $d\eta$ is the Haar measure on $\bF$ self dual with respect to the pairing $\langle\xi,\eta\rangle=\psi(\xi\eta)$
  (and hence $|r|_\bF^{1/2}d\eta$ is self dual with respect to the pairing $\psi_r(\xi\eta)$).
%  $\Gamma_\psi\left(-\frac{rv}{\epsilon}\right)$ is the absolute Weil index of the quadratic character
%  $\bF\ni z\mapsto\psi\left(-\frac{rv}{\epsilon}z^2/2\right)$, and $\zeta(u+v\sqrt{\epsilon})=\pm1$ is
%  defined in~\cite[Theorem 5.7]{A} (in fact, there are~2 possibilities for $\zeta$ depending on additional choices).
  Furthermore, for $v=0$, $u=1$ we
  have $A_{1,0}\phi(\xi)=\phi(\xi)$, and for $v=0$, $u=-1$ we have 
  $A_{-1,0}\phi(\xi)=(-1,-\epsilon)_\bF\zeta(-1)\gamma\!\left(y\mapsto\psi\left(\frac{ry^2}{2}\right)\right)^2\phi(-\xi)$.
\end{lem}

\begin{proof}
Let $R_g$ denote the displayed operator without the factor $\zeta(g)$ for
$g=\begin{pmatrix}u&\epsilon v\\v&u\end{pmatrix}$ with $v\ne0$.
Direct substitution for translations and modulations gives
\[
R_g\rho_r(a,b;0)R_g^{-1}=\rho_r(ua+\epsilon vb,va+ub;0).
\]
Here we use the Heisenberg action stated above and write $\rho_r=\rho$.
Recall the absolute Weil-index identities
\[\gamma\!\left(y\mapsto\psi\left(\frac{as^2y^2}{2}\right)\right)=\gamma\!\left(y\mapsto\psi\left(\frac{ay^2}{2}\right)\right),\]
\[\gamma\!\left(y\mapsto\psi\left(\frac{-ay^2}{2}\right)\right)=\gamma\!\left(y\mapsto\psi\left(\frac{ay^2}{2}\right)\right)^{-1},\]
\[\gamma\!\left(y\mapsto\psi\left(\frac{ay^2}{2}\right)\right)^2=\gamma\!\left(y\mapsto\psi\left(\frac{y^2}{2}\right)\right)^2(a,-1)_\bF.\]
In particular $\gamma\!\left(y\mapsto\psi\left(\frac{-rvy^2}{2\epsilon}\right)\right)=\gamma\!\left(y\mapsto\psi\left(\frac{-r\epsilon vy^2}{2}\right)\right)$.
For self-dual measure, completing the square gives
\begin{equation}\label{gaussian-review}
\operatorname{Os}\!\!\int_\bF\psi(ax^2/2+bx)\,dx
=\gamma\!\left(y\mapsto\psi\left(\frac{ay^2}{2}\right)\right)|a|_\bF^{-1/2}\psi(-b^2/(2a)),\qquad a\ne0.
\end{equation}
The oscillatory integral is defined by stabilization over expanding compact
open balls. Also $(-1,-1)_\bF=1$ because the residue characteristic is odd.

Define
\[
x(g)=
\begin{cases}v,&v\ne0,\\u,&v=0,\end{cases}
\qquad
c(g,h)=(x(g)x(gh),x(h)x(gh))_\bF.
\]
This is the restriction to the torus of Adams' cocycle
\cite[\S2, equation (10)]{A}.
Complete the definition of $R_g$ by
\begin{equation}\label{exceptional}
R_1=1,\qquad
R_{-1}=\kappa P,\qquad
\kappa=(-1,-\epsilon)_\bF\gamma\!\left(y\mapsto\psi\left(\frac{ry^2}{2}\right)\right)^2,\quad
(P\phi)(\xi)=\phi(-\xi).
\end{equation}
Then
\begin{equation}\label{cocycle}
R_gR_h=c(g,h)R_{gh}.
\end{equation}

Here is a direct check. Suppose first that $v_g,v_h,v_{gh}$ are
nonzero, and put $B_i=\epsilon v_i$ for $i=g,h,gh$.
Integration in the intermediate variable in the product of the
two kernels gives the scalar
\[
\frac{\gamma\!\left(y\mapsto\psi\left(\frac{-rB_{gh}y^2}{2}\right)\right)\,
\gamma\!\left(y\mapsto\psi\left(\frac{-rB_{gh}y^2}{2B_gB_h}\right)\right)}
{\gamma\!\left(y\mapsto\psi\left(\frac{-rB_gy^2}{2}\right)\right)\gamma\!\left(y\mapsto\psi\left(\frac{-rB_hy^2}{2}\right)\right)}.
\]
The Weil-index identities reduce this expression to
\[
(B_gB_{gh},B_hB_{gh})_\bF
=(v_gv_{gh},v_hv_{gh})_\bF=c(g,h).
\]
The absolute-value factors also agree with those of $R_{gh}$.

If $h=g^{-1}$ with $v_g\ne0$, the kernels compose to the identity,
since $\gamma\!\left(y\mapsto\psi\left(\frac{ay^2}{2}\right)\right)\gamma\!\left(y\mapsto\psi\left(\frac{-ay^2}{2}\right)\right)=1$. This agrees with
$c(g,g^{-1})=(-v_g,v_g)_\bF=1$.
If $h=-g^{-1}$, they instead compose to
\[
R_gR_h=\gamma\!\left(y\mapsto\psi\left(\frac{-r\epsilon v_gy^2}{2}\right)\right)^{-2}P.
\]
On the other hand,
\[
c(g,-g^{-1})=(-v_g,-v_g)_\bF=(-v_g,-1)_\bF,
\]
and the identities above give
\[
\gamma\!\left(y\mapsto\psi\left(\frac{-r\epsilon v_gy^2}{2}\right)\right)^{-2}
=(-v_g,-1)_\bF\,\kappa.
\]
This proves \eqref{cocycle} also when $gh=-1$.
Finally, $P$ commutes with every torus kernel, and comparison of
$R_g$ with $R_{-g}$ gives
$R_{-1}R_g=(-1,v_g)_\bF R_{-g}$ for $v_g\ne0$.
Together with $\kappa^2=1$, this handles the cases in which an
individual factor equals $-1$; the identity cases are immediate.

Consequently, if $\zeta$ is the $\{\pm1\}$-valued splitting function
of \cite[Theorem 5.7]{A} for this embedding, so that
\[
\zeta(g)\zeta(h)=c(g,h)\zeta(gh),
\]
then $A_g=\zeta(g)R_g$ satisfies $A_gA_h=A_{gh}$.
%In particular the factor $(-1,-\epsilon)_\bF$ in the manuscript's formula at $-1$ must be retained.
This verifies explicitly the convention matching with the Weil operators of~\cite{R} and the splitting of~\cite{A}.\footnote{Although
Adams assumes characteristic zero, the construction in~\S5, specifically~Lemmas~5.5-5.6 and~Theorem~5.7 in the case $(-1,-1)_{\mathbf F}=1$ apply
unchanged to non-Archimedean local fields of odd positive characteristic. Taking $\alpha=1$ gives a $\{\pm1\}$-valued splitting.
Together with the Weil-index identities and the
preceding kernel calculation, this proves the lemma over every non-Archimedean local field of odd residue characteristic.}
\end{proof}

\subsection{$\Gamma$-function, unramified}
\label{gamma unr}
Given a multiplicative character $\omega\colon\CO^\times_n\to\BC^\times$, in the odd case we can restrict $\omega$ to
$\CO^\times_{2m-1}(m)\cong\CO_{m-1} : \CO_{m-1}\ni a\mapsto1+t^m a$ to obtain an additive character of $\CO_{m-1} :
\omega(1+t^m a)=\psi_{m-1}(c_\omega\cdot a)$ for certain $c_\omega\in\CO_{m-1}$ (not necessarily invertible). In the even case,
similarly for $ a\in\CO_{m-1}$ we have $\omega(1+t^{m+1} a)=\psi_{m-1}(c_\omega\cdot a)$.

We may view $\omega$ as a function $\omega\in\BC[\CO^\times_n]$ in the full conjugate dual space
of $\sigma_n=\ind_{\CO_n}^{P(\CO)_n}(\psi_n)$. (Recall that $\sigma_n$ is realized in the Schwartz space $\CS(\CO^\times_n)$, and the full conjugate dual
is realized in the dual Schwartz space $\CS'(\CO^\times_n)$.)
Using the invariant unitary structure on $\CS(\CO^\times_n)$ we embed $\CS(\CO^\times_n)$ into $\CS'(\CO^\times_n)$ and
extend the representation $\pi_\theta$ of $G(\CO_n)$ on $\CS(\CO^\times_n)$ to the representation $\pi_\theta^\wedge$ on $\CS'(\CO^\times_n)$.
We apply
$\pi_\theta^\wedge\begin{pmatrix}0&1\\ -1&0\end{pmatrix}$ to $\omega\in\ind_{\CO_n}^{P(\CO)_n}(\psi_n)^\wedge=\Res_{G(\CO_n)}^{P(\CO)_n}\pi_\theta^\wedge$.
We get $\omega^{-1}$ up to a constant multiple $\Gamma(\pi_\theta,\omega;\psi_n)\in\BC^\times$.

In the following theorem we calculate $\Gamma(\pi_\theta,\omega;\psi_n)$ explicitly.
Recall the notation introduced right before~Lemma~\ref{trivial}: $c\in\CO^\times_n$ is defined via the
identity $\varPsi(y)=\psi_n(cy)$ for any $y\in\CO_n$. We lift $\frac{c_\omega}{c_\theta}\in\CO_{m-1}$ to an element of $\CO_n$ represented
by a polynomial of degree $<m$. Put $q=c_\omega/c_\theta$, $r=c_\theta(0)$, and $c_0=q(0)$. In the even case $n=2m$ we define $b\in\bF$ by
$\varPsi(\xi t^m q)=\psi(b\xi)$. The notation $\gamma$ for the absolute Weil index and $\zeta$ for Adams' factor was introduced right
before~Lemma~\ref{RRA}.
Finally, the restriction of $\omega$ to the truncated exponential subgroup
defines a unique $\ell\in\bF$ such that
\[
\omega(1+t^m\xi+\tfrac12t^{2m}\xi^2)=\psi(\ell\xi).
\]
Here $n=2m\ge2$, and all expressions are taken modulo $t^{2m+1}$.
Using the definition of $c_\omega$ on the deeper subgroup gives
\[
\omega(1+t^m\xi)
=\psi\left(\ell\xi-\frac{rc_0}{2}\xi^2\right).
\]

\begin{thm}
  \label{compute gamma eps}
  \textup{(1)} For $n=2m-1$, \[\Gamma(\pi_\theta,\omega;\psi_n)
  =\theta\left(\frac{c_\omega}{c_\theta}+\sqrt{\epsilon}^{-1}\right)\omega\left(\frac{c^2_\omega}{c^2_\theta}-\epsilon^{-1}\right)\omega(c)^2.\]
  \textup{(2)} For $n=0$, the character $\theta$ (resp.\ $\omega$) is the inflation of the same named character $\theta\colon T^\bE\to\BC^\times$
  (resp.\ $\omega\colon\bF^\times\to\BC^\times$), and the representation $\pi_\theta$ is the inflation of a cuspidal representation $\Pi$
  of $G(\bF)$. Then
  $\Gamma(\pi_\theta,\omega;\psi_0)$ is nothing but the classical $\Gamma$-function $\Gamma(\Pi,\omega;\psi)$ of~\S\ref{Gamma classical}.
  
  \textup{(3)} For $n=2m>0$, we get for $c_0\ne0$
 \begin{multline*}
%  \Gamma(\pi_\theta,\omega;\psi_n)
%  =\omega(c)^2\theta\left(\frac{c_\omega}{c_\theta}+\sqrt{\epsilon}^{-1}\right)  
%  \zeta\left(\frac{c_0^2+\epsilon^{-1}}{c_0^2-\epsilon^{-1}}
%  -\frac{2c_0}{\sqrt{\epsilon}(c_0^2-\epsilon^{-1})}\right)\\
%  \times\Gamma_\psi\left(-\frac{2c_\theta(0)c_0}{\epsilon(c_0^2-\epsilon^{-1})}\right)^{-1}
%  \left|\frac{-2c_0}{\epsilon(c_0^2-\epsilon^{-1})}\right|_\bF^{-1/2}|c_\theta(0)|_\bF^{1/2}\\
%  \times\int_{\xi\in\bF}\psi\left(-c_\theta(0)\frac{c_0^2\epsilon+1}{4c_0}\xi^2\right)
%  \varPsi\left(-\xi t^m\frac{c_\omega}{c_\theta}\right)\omega\left(\frac{c_\omega^2}{c_\theta^2}(1+t^m\xi)^2-\epsilon^{-1}\right)\omega(1+t^m\xi)^{-1}d\xi,
  \Gamma(\pi_\theta,\omega;\psi_n)
  =\theta\left(\frac{c_\omega}{c_\theta}+\sqrt{\epsilon}^{-1}\right)\omega\left(\frac{c^2_\omega}{c^2_\theta}-\epsilon^{-1}\right)\omega(c)^2\\ 
  \times\zeta\left(\frac{c_0^2+\epsilon^{-1}}{c_0^2-\epsilon^{-1}}
  -\frac{2c_0}{\sqrt{\epsilon}(c_0^2-\epsilon^{-1})}\right)
  \gamma\!\left(y\mapsto\psi\left(\frac{c_\theta(0)c_0y^2}{(\epsilon^2c_0^2-\epsilon)}\right)\right)^{-2}
  \psi\left(\frac{(\ell-b)^2c_0}{c_\theta(0)(c_0^2-\epsilon^{-1})}\right),
  %\left|\frac{c_\theta(0)(c_0^2-\epsilon^{-1})}{2c_0}\right|_\bF^{1/2}
% \\
% \times\int_{\xi\in\bF}\psi\left(c_\theta(0)\frac{c_0^2+\epsilon^{-1}}{4c_0}\xi^2\right)
% \varPsi\left(-\xi t^m\frac{c_\omega}{c_\theta}\right)
% \omega(1+t^m\xi)
% d\xi,
  \end{multline*}
  and for $c_0=0$,  $\Gamma(\pi_\theta,\omega;\psi_n)$
  \[=\theta\left(\frac{c_\omega}{c_\theta}+\sqrt{\epsilon}^{-1}\right)\omega\left(\frac{c^2_\omega}{c^2_\theta}-\epsilon^{-1}\right)\omega(c)^2
  (-1,-\epsilon)_\bF\zeta(-1)\gamma\!\left(y\mapsto\psi\left(\frac{c_\theta(0)y^2}{2}\right)\right)^2.\]
\end{thm}

\begin{proof}
  From~Lemma~\ref{trivial} it follows easily that
  $\Gamma(\pi_\theta,\omega;\psi_n)=\omega(c)^2\Gamma(\pi_\theta,\omega,\varPsi)$.
  %(i.e.\ the $\Gamma$-function is independent of the choice of an additive character $\Psi$ of $\CO_n$ provided that $\Psi|_{\CO_n(n)}\not\equiv1$).  
  We start with the function $\phi_\omega\in\ind_{\CO_n}^{P(\CO)_n}(\varPsi)^\wedge$,
  $\phi_\omega\begin{pmatrix}x&z\\ 0&1\end{pmatrix}=\varPsi(z)\omega(x)$.

  (1) We consider the case of odd $n=2m-1$.
  According to~Lemma~\ref{intertwiners},
  \begin{multline*}T\phi_\omega\begin{pmatrix}x&z\\ 0&1\end{pmatrix}
  =\int_{ a\in\CO_{m-1}}\phi_\omega\left(\begin{pmatrix}1+t^m a&0\\ 0&1\end{pmatrix}\begin{pmatrix}x&z\\ 0&1\end{pmatrix}\right)d a\\
  =\int_{ a\in\CO_{m-1}}\varPsi((1+t^m a)z)\omega((1+t^m a)x)d a\\
  =\int_{ a\in\CO_{m-1}}\psi_{m-1}(c_\theta\cdot a\cdot z_{\leq m-1})\psi_{m-1}(c_\omega\cdot a)\varPsi(z_{\geq m})\omega(x)d a\\
  =|c(0)|_\bF^{-\frac{m}{2}}\varPsi(z_{\geq m})\omega(x)\delta_{-c_\omega/c_\theta}(z_{\leq m-1}),
  \end{multline*}
  where $c(0)\in\bF^\times$ is the constant term of $c\in\CO^\times_n$ (the first factor comes from our normalization of the
  Haar measure in the definition of the intertwining operator $T$), $\delta_{-c_\omega/c_\theta}(?)$ is the Dirac delta-function at
  $-c_\omega/c_\theta\in\CO_{m-1}$, and for $z=z_0+z_1t+\ldots+z_{2m-1}t^{2m-1}$ we set
  $z_{\leq m-1}=z_0+\ldots+z_{m-1}t^{m-1}$, while $z_{\geq m}=z_mt^m+\ldots+z_{2m-1}t^{2m-1}$.
  Now 
\begin{equation}
  \label{identity}
\begin{pmatrix}x&z\\ 0&1\end{pmatrix}\begin{pmatrix}0&1\\ -1&0\end{pmatrix}=\begin{pmatrix}\frac{-xz}{z^2-\epsilon^{-1}} &  
\frac{-\epsilon^{-1}x}{z^2-\epsilon^{-1}}\\ \frac{-x}{z^2-\epsilon^{-1}} & \frac{-xz}{z^2-\epsilon^{-1}}\end{pmatrix}
\begin{pmatrix}\frac{z^2-\epsilon^{-1}}{x}&-z\\ 0&1\end{pmatrix}.
\end{equation}
In the RHS, the first matrix represents the element
$\frac{-xz}{z^2-\epsilon^{-1}}+\frac{-x\epsilon^{-1}}{z^2-\epsilon^{-1}}\sqrt{\epsilon}\in\CE^\times$.
We have $\theta(\frac{-xz}{z^2-\epsilon^{-1}}+\frac{-x\epsilon^{-1}}{z^2-\epsilon^{-1}}\sqrt{\epsilon})=
\theta(z+\sqrt{\epsilon}^{-1})$ since $\theta\colon\CE^\times\to\BC^\times$ factors through $\CE^\times/\CK^\times$.
It follows that
\begin{multline*}
  \pi_\theta^\wedge\begin{pmatrix}0&1\\ -1&0\end{pmatrix} T\phi_\omega\begin{pmatrix}x&z\\ 0&1\end{pmatrix}
  =\theta(z+\sqrt{\epsilon}^{-1})T\phi_\omega\begin{pmatrix}\frac{z^2-\epsilon^{-1}}{x}&-z\\ 0&1\end{pmatrix}\\
  =|c(0)|_\bF^{-\frac{m}{2}}\varPsi(-z_{\geq m})\theta(z+\sqrt{\epsilon}^{-1})\omega(z^2-\epsilon^{-1})\omega(x)^{-1}\delta_{-c_\omega/c_\theta}(-z_{\leq m-1}).   
% =\theta(z+\sqrt{\epsilon^{-1}})\int_{ a\in\CO_{m-1}}\varPsi(-(1+t^m a)z)\omega((1+t^m a)(z^2-\epsilon^{-1})x^{-1})d a.
  \end{multline*}
  It remains to apply the inverse intertwining operator to the resulting function, and compute its value at $\Id\in P(\CO)_n$.
  From~Lemma~\ref{intertwiners}, we obtain the desired value (the prefactor $|c(0)|_\bF^{-\frac{m}{2}}$ cancels out due to the choice
  of the Haar measure in the definition of the intertwining operator $T^{-1}$).
  %\[\int\int_{ b\in\CO_{m-1},\  a\in\CO_{m-1}}
  %\theta( b-\sqrt{\epsilon}^{-1})\varPsi((1+t^m a) b)\omega((1+t^m a)( b^2-\epsilon^{-1}))d a d b.\]
  %Finally, we replace $\varPsi((1+t^m a) b)\omega((1+t^m a)$ by
  %$\psi_{m-1}(c_\theta\cdot a b)\psi_{m-1}(c_\omega\cdot a)$, and integrate over $ a$. We get the delta-function of
  %$\frac{-c_\omega}{c_\theta}$.

  (2) is obvious.
  
  (3) In the even case $n=2m$, we still have
  \[T\phi_\omega\begin{pmatrix}x&z\\ 0&1\end{pmatrix}=|c(0)|_\bF^{-\frac{m}{2}}f\begin{pmatrix}x&z\\ 0&1\end{pmatrix}
  :=|c(0)|_\bF^{-\frac{m}{2}}\varPsi(z_{\geq m})\omega(x)\delta_{-c_\omega/c_\theta}(z_{\leq m-1}).\] But the calculation of
  $\pi_\theta^\wedge\begin{pmatrix}0&1\\ -1&0\end{pmatrix} T\phi_\omega\begin{pmatrix}x&z\\ 0&1\end{pmatrix}=
  \pi_\theta^\wedge(z+\sqrt{\epsilon}^{-1})T\phi_\omega\begin{pmatrix}\frac{z^2-\epsilon^{-1}}{x}&-z\\ 0&1\end{pmatrix}$
  is more complicated. Following~\cite[\S5.6.4]{BK}, in order to calculate the action of
  $\pi_\theta^\wedge(z+\sqrt{\epsilon}^{-1})$, we
  consider the projection of $z+\sqrt{\epsilon}^{-1}\in T^\CE=T^\bE\times T^\CE(1)$ to $T^\bE$.
  Namely, for $z=z_0+z_1t+\ldots+z_{2m}t^{2m}$, this projection equals $z_0+\sqrt{\epsilon}^{-1}$.
  Then $\pi_\theta^\wedge(z+\sqrt{\epsilon}^{-1})$ acts via the character $\theta(z+\sqrt{\epsilon}^{-1})$ times
  the action of $z_0+\sqrt{\epsilon}^{-1}$ in the quasi-Heisenberg representation $\CW_\theta$ via a lift
  $s\colon T^\bE\cong\bE^1\to\on{Mp}(2,\bF)$.
  %have to decompose $z+\sqrt{\epsilon}^{-1}\in T^\CE$ according to
  %$T^\CE=T^\bE\times T^\CE(1)$. Namely, for $z=z_0+z_1t+\ldots+z_{2m}t^{2m}$, we have
  %$z+\sqrt{\epsilon}^{-1}=(z_0+\sqrt{\epsilon}^{-1})\cdot\frac{z+\sqrt{\epsilon}^{-1}}{z_0+\sqrt{\epsilon}^{-1}}$.
  %Then the second factor acts via the character $\theta\left(\frac{z+\sqrt{\epsilon}^{-1}}{z_0+\sqrt{\epsilon}^{-1}}\right)$,
  %while the first factor acts in the quasi-Heisenberg representation $W_\theta$ via a lift
  %$s\colon T^\bE\cong\bE^1\to\on{Mp}(2,\bF)$.
  According to~Lemma~\ref{TE}, $z_0+\sqrt{\epsilon}^{-1}\in T^\bE$ corresponds to
  $u+v\sqrt{\epsilon}\in\bE^1$, where $u=\frac{z_0^2+\epsilon^{-1}}{z_0^2-\epsilon^{-1}}$, $v=\frac{-2z_0}{\epsilon(z_0^2-\epsilon^{-1})}$.
  Hence by~Lemma~\ref{banal},
  \begin{multline*}
  S\left(\pi_\theta^\wedge\begin{pmatrix}0&1\\ -1&0\end{pmatrix} T\phi_\omega\begin{pmatrix}x&z\\ 0&1\end{pmatrix}\right)=\\
  |c(0)|_\bF^{-\frac{m}{2}}\theta(z+\sqrt{\epsilon}^{-1})\cdot  
  %|c(0)|_\bF^{-\frac{m}{2}}\theta\left(\frac{z+\sqrt{\epsilon}^{-1}}{z_0+\sqrt{\epsilon}^{-1}}\right)\cdot
  (u+v\sqrt{\epsilon})\left((Sf)\begin{pmatrix}\frac{z^2-\epsilon^{-1}}{x}&-z\\ 0&1\end{pmatrix}(\xi)\right),
  %(u+v\sqrt{\epsilon})\Big( S\big(\varPsi(-z_{\geq m})\omega(z^2-\epsilon^{-1})\omega(x)^{-1}\delta_{-c_\omega/c_\theta}(-z_{\leq m-1})\big)\Big),
  \end{multline*}
  where $S$ is the intertwining operator of~Lemma~\ref{banal}, so that
  \begin{multline}
    \label{swing}
%  S\big(\varPsi(-z_{\geq m})\omega(z^2-\epsilon^{-1})\omega(x)^{-1}\delta_{-c_\omega/c_\theta}(-z_{\leq m-1})\big)\\
%  =\varPsi(-z_{\geq m})\varPsi(-\xi t^m z_{\leq m})\omega\left(z^2(1+t^m\xi)^2-\epsilon^{-1}\right)\omega(x)^{-1}\omega(1+t^m\xi)^{-1}
%  \delta_{c_\omega/c_\theta}(z_{\leq m-1})
    (Sf)\begin{pmatrix}\frac{z^2-\epsilon^{-1}}{x}&-z\\ 0&1\end{pmatrix}(\xi)=f\left(\begin{pmatrix}1+t^m\xi&0\\ 0&1\end{pmatrix}
        \begin{pmatrix}\frac{z^2-\epsilon^{-1}}{x}&-z\\ 0&1\end{pmatrix}\right)\\
  =\varPsi(-z_{\geq m})\varPsi(-\xi t^m z_{\leq m})\omega(z^2-\epsilon^{-1})
  \omega(x)^{-1}\omega(1+t^m\xi)
  \delta_{c_\omega/c_\theta}(z_{\leq m-1})
  \end{multline}
  (as a function of $\xi\in\bF$).
  %Indeed, denoting for short
  %$h(x,z):=\varPsi(-z_{\geq m})\omega(z^2-\epsilon^{-1})\omega(x)^{-1}\delta_{-c_\omega/c_\theta}(-z_{\leq m-1})$, by the definition
  %of $S$ in~Lemma~\ref{banal}, $(Sh)(x,z)(\xi)=h((1+t^m\xi)x,(1+t^m\xi)z)$, and direct substitution gives~\eqref{swing}.

  Now the action of $u+v\sqrt{\epsilon}$ in the Weil representation $\CW_\theta$ is given by~Lemma~\ref{RRA}. Namely, it is realized
  in the Schr\"odinger model $\CS(\bF)$ of an irreducible Heisenberg representation. Here $\bF$ has coordinate corresponding to the
  first coordinate in the tautological 2-dimensional representation of $\Sp(2,\bF)$. If this tautological representation is
  realized in $\on{Ker}(p\colon \fg(\CO)_m(m)\to\ft^\CE_m(m))$ as in~Lemma~\ref{TE}, this is the first coordinate in the basis
  $\begin{pmatrix}\frac{t^m}{2}&0\\ 0&\frac{-t^m}{2}\end{pmatrix},
  \begin{pmatrix}0&\frac{t^m}{2}\\ \frac{-\epsilon t^m}{2}&0\end{pmatrix}$
  Note that the image of these basis elements %$\begin{pmatrix}\frac{t^m}{2}&0\\ 0&\frac{-t^m}{2}\end{pmatrix},
  %\begin{pmatrix}0&\frac{t^m}{2}\\ \frac{-\epsilon t^m}{2}&0\end{pmatrix}$ used in~Lemma~\ref{TE}
  under projection to the Lie algebra $\fp(\CO)_m(m)$ along $\ft^\CE_m(m)$
  are exactly the logarithms of $\begin{pmatrix}1+t^m&0\\ 0&1\end{pmatrix},
  \begin{pmatrix}1&t^m\\ 0&1\end{pmatrix}$. Hence the above first coordinate is nothing but the coordinate $\xi$ in
  $\begin{pmatrix}1+t^m\xi&0\\ 0&1\end{pmatrix}$.    
  We only have to determine the central character $\psi_r$ determining the irreducible Heisenberg representation $\CW_\theta$
  appearing in~Lemma~\ref{RRA}.    
  Since the commutator of $\begin{pmatrix}1+t^m\xi&0\\ 0&1\end{pmatrix}$ and
  $\begin{pmatrix}1&t^m\eta\\ 0&1\end{pmatrix}$ equals $\begin{pmatrix}1&t^{2m}\xi\eta\\ 0&1\end{pmatrix}$, and its character in the
  quasi-Heisenberg representation $\CW_\theta$ is $\varPsi(t^{2m}\xi\eta)=\psi(c_\theta(0)\xi\eta)$, we conclude that $r=c_\theta(0)$.

  Applying the inverse intertwining operator $S^{-1}$ of~Lemma~\ref{banal}, we obtain for $v\ne0$
  \begin{multline}
    \label{Adams}
    %S^{-1}(u+v\sqrt{\epsilon})\Big( S\big(\varPsi(-z_{\geq m})\omega(z^2-\epsilon^{-1})\omega(x)^{-1}\delta_{-c_\omega/c_\theta}(-z_{\leq m-1})\big)\Big)\\
  S^{-1}(u+v\sqrt{\epsilon})\left((Sf)\begin{pmatrix}\frac{z^2-\epsilon^{-1}}{x}&-z\\ 0&1\end{pmatrix}(\xi)\right)\\ 
  =\varPsi(-z_{\geq m})\omega(x)^{-1}\delta_{c_\omega/c_\theta}(z_{\leq m-1})\zeta(u+v\sqrt{\epsilon})
  \gamma\!\left(y\mapsto\psi\left(-\frac{c_\theta(0)vy^2}{2\epsilon}\right)\right)^{-1}\left|\frac{c_\theta(0)}{\epsilon v}\right|_\bF^{1/2}\\
%  \times\int_{\xi\in\bF}\psi\left(c_\theta(0)\frac{u\xi^2}{2v}\right)
%  \varPsi(-\xi t^m z_{\leq m})\omega\left(z^2(1+t^m\xi)^2-\epsilon^{-1}\right)\omega(1+t^m\xi)^{-1}d\xi,
  \times\int_{\xi\in\bF}\psi\left(-c_\theta(0)\frac{u\xi^2}{2\epsilon v}\right)
  \varPsi(-\xi t^m z_{\leq m})\omega(z^2-\epsilon^{-1})
  \omega(1+t^m\xi)
  d\xi,
  \end{multline}
  and for $v=0$, $u=-1$
  \begin{multline*}
  S^{-1}(-1)\Big( S\big(\varPsi(-z_{\geq m})\omega(z^2-\epsilon^{-1})\omega(x)^{-1}\delta_{-c_\omega/c_\theta}(-z_{\leq m-1})\big)\Big)=\\
  \varPsi(-z_{\geq m})\omega(x)^{-1}\omega(z^2-\epsilon^{-1})\delta_{c_\omega/c_\theta}(z_{\leq m-1})
  (-1,-\epsilon)_\bF\zeta(-1)\gamma\!\left(y\mapsto\psi\left(\frac{c_\theta(0)y^2}{2}\right)\right)^2.
  \end{multline*}
  Now we plug in the values of $u,v$, and apply the inverse intertwining operator $T^{-1}$ of~Lemma~\ref{intertwiners}.
  Setting $c_0=\frac{c_\omega}{c_\theta}(0)$, we obtain
\begin{multline*}
%  \Gamma(\pi_\theta,\omega;\psi_n)
%  =\omega(c)^2\theta\left(\frac{c_\omega}{c_\theta}+\sqrt{\epsilon}^{-1}\right)  
%  \zeta\left(\frac{c_0^2+\epsilon^{-1}}{c_0^2-\epsilon^{-1}}
%  -\frac{2c_0}{\sqrt{\epsilon}(c_0^2-\epsilon^{-1})}\right)\\
%  \times\Gamma_\psi\left(-\frac{2c_\theta(0)c_0}{\epsilon(c_0^2-\epsilon^{-1})}\right)^{-1}
%  \left|\frac{-2c_0}{\epsilon(c_0^2-\epsilon^{-1})}\right|_\bF^{-1/2}|c_\theta(0)|_\bF^{1/2}\\
%  \times\int_{\xi\in\bF}\psi\left(-c_\theta(0)\frac{c_0^2\epsilon+1}{4c_0}\xi^2\right)
%  \varPsi\left(-\xi t^m\frac{c_\omega}{c_\theta}\right)\omega\left(\frac{c_\omega^2}{c_\theta^2}(1+t^m\xi)^2-\epsilon^{-1}\right)\omega(1+t^m\xi)^{-1}d\xi,
  \Gamma(\pi_\theta,\omega;\psi_n)
  =\theta\left(\frac{c_\omega}{c_\theta}+\sqrt{\epsilon}^{-1}\right)\omega\left(\frac{c^2_\omega}{c^2_\theta}-\epsilon^{-1}\right)\omega(c)^2\\ 
  \times\zeta\left(\frac{c_0^2+\epsilon^{-1}}{c_0^2-\epsilon^{-1}}
  -\frac{2c_0}{\sqrt{\epsilon}(c_0^2-\epsilon^{-1})}\right)
  \gamma\!\left(y\mapsto\psi\left(\frac{c_\theta(0)c_0y^2}{(\epsilon^2c_0^2-\epsilon)}\right)\right)^{-1}
  \left|\frac{c_\theta(0)(c_0^2-\epsilon^{-1})}{2c_0}\right|_\bF^{1/2}
  \\
  \times\int_{\xi\in\bF}\psi\left(c_\theta(0)\frac{c_0^2+\epsilon^{-1}}{4c_0}\xi^2\right)
  \varPsi\left(-\xi t^m\frac{c_\omega}{c_\theta}\right)
  \omega(1+t^m\xi)
  d\xi,
  \end{multline*}
  and for $c_0=0$,  $\Gamma(\pi_\theta,\omega;\psi_n)$
  \[=\theta\left(\frac{c_\omega}{c_\theta}+\sqrt{\epsilon}^{-1}\right)\omega\left(\frac{c^2_\omega}{c^2_\theta}-\epsilon^{-1}\right)\omega(c)^2
  (-1,-\epsilon)_\bF\zeta(-1)\gamma\!\left(y\mapsto\psi\left(\frac{c_\theta(0)y^2}{2}\right)\right)^2.\]
Set $q=c_\omega/c_\theta$, $r=c_\theta(0)$, and $D=c_0^2-\epsilon^{-1}$.

The integral should be described as
an oscillatory integral, defined by stabilization over expanding
compact open balls. In general it is not absolutely convergent.

Recall $\ell\in\bF$ and $b\in\bF$ defined before the theorem. The integral can be rewritten as
\[
\operatorname{Os}\!\!\int_{\bF}\psi(A\xi^2+(\ell-b)\xi)\,d\xi,
\qquad
A=\frac{r(\epsilon^{-1}-c_0^2)}{4c_0}.
\]
For $c_0\ne0$ this coefficient is nonzero. Formula
\eqref{gaussian-review} evaluates the integral as
\[
\gamma\!\left(y\mapsto\psi\left(Ay^2\right)\right)|2A|^{-1/2}
\psi\left(-\frac{(\ell-b)^2}{4A}\right).
\]
%It is nonzero. Moreover,
Finally, $|2A|=|rD/(2c_0)|$, so its absolute value factor cancels
the absolute value factor outside the integral, while $\gamma\!\left(y\mapsto\psi\left(Ay^2\right)\right)$ contributes
the same inverse absolute Weil index outside the integral. Indeed, setting
$B=rc_0/(\epsilon^2D)$, we have
\[
 A=-B\left(\frac{\epsilon D}{2c_0}\right)^2,
 \qquad
 \gamma(y\mapsto\psi(Ay^2))=\gamma(y\mapsto\psi(By^2))^{-1}.
\]
Thus the product of the two Weil-index factors is
$\gamma(y\mapsto\psi(By^2))^{-2}$, as stated.
%This provides a further normalization check.
\end{proof}

\begin{rem}
  Recall from right above~Lemma~\ref{TE} that the indexation of irreducible representations $\pi_\theta$ of even depth $2m$
  depended on a choice of splitting $s\colon T^\bE\to\on{Mp}(2)$. According to~\cite[\S\S5.6.5,5.6.6]{BK}, in case $\bE$ is an
  unramified extension of $\bF$, there is a canonical choice of $s$. However, in case $\bE$ is ramified, there are two splittings $s,s'$
  corresponding to two choices of Adams' factor $\zeta$ in~Lemma~\ref{RRA},
  and the change of splitting changes the indexation $\pi_\theta^s\simeq\pi_{\theta\cdot\chi}^{s'}$. Here $\chi$ is the unique nontrivial
  character $T^\CE_{2m}\twoheadrightarrow T^\bE\twoheadrightarrow T^\bE/(T^\bE)^2\iso\{\pm1\}$. We conclude that
  $\displaystyle{\Gamma(\pi^{s'}_\theta,\omega;\psi_n)=\chi\left(\frac{c_\omega}{c_\theta}+\sqrt{\epsilon}^{-1}\right)
  \Gamma(\pi^s_\theta,\omega;\psi_n)}$.
\end{rem}

\subsection{Representations of $G'(\CO)$}
\label{Rep G'O}
We follow the notation of~\cite[\S6]{BK}. Let $\CE$ be a ramified
quadratic extension of $\CK$, namely $\CE=\CK(\sqrt{-t})$ for a uniformizer $t\in\CO$ (note that we may have changed the
parameter $t$ in $\CK=\bF(\!(t)\!)$). We fix an embedding
$T^\CE=\CE^\times/\CK^\times\hookrightarrow\PGL(2,\CK)=G(\CK)$ with the image consisting of the classes of the
elements $\begin{pmatrix}a&b\\ -bt&a\end{pmatrix}$. Note that the image actually lies in the bounded subgroup
$G'(\CO)\subset G(\CK)$ generated by $\begin{pmatrix}0&1\\ -t&0\end{pmatrix}$ and the Iwahori subgroup
$I=\left\{\begin{pmatrix}a&b\\ c&d\end{pmatrix} : a,d\in\CO^\times,\ b\in\CO,\ c\in t\CO\right\}\subset G'(\CO)$.
We have the $\on{Ad}_{T^\CE}$-invariant projection $p\colon\fg(\CK)\to\ft^\CE$,
$\begin{pmatrix}a&b\\ c&-a\end{pmatrix}\mapsto\begin{pmatrix}0&\frac{b-ct^{-1}}{2}\\\frac{c-bt}{2}&0\end{pmatrix}$.
Under the natural identifications, $p$ induces the same named map from $G'(\CO)_{n+\frac12}(\frac{n+1}{2})$ to
$T^\CE_{n+\frac12}(\frac{n+1}{2})$, and the latter map is a homomorphism. Here
\[T^\CE_{n+\frac12}=\BZ/2\BZ\times\{a_0+a_1\sqrt{-t}+\ldots+a_{2n+1}\sqrt{-t}^{2n+1}+a_{2n+2}t^{n+1}\}/\{a_0+a_2t+\ldots+a_{2n+2}t^{n+1}\},\]
and $a_0\in\bF^\times$, $a_{>0}\in\bF$, while $\BZ/2\BZ$ is generated by the class of $\sqrt{-t}$.

Given a character $\theta\colon T^\CE_{n+\frac12}\to\BC^\times$ of depth $n+\frac12$,
we can consider the character $\theta\circ p$ of the normal abelian subgroup $G'(\CO)_{n+\frac12}(\frac{n+1}{2})$, and its stabilizer
in $G'(\CO)_{n+\frac12}$ is $T^\CE_{n+\frac12}\cdot G'(\CO)_{n+\frac12}(\frac{n+1}{2})$. We again assume that $\bar\theta\ne\theta$.
Then the induced representation $\pi_\theta=\ind_{\on{Stab}(\theta\circ p)}^{G'(\CO)_{n+\frac12}}(\theta)$ is irreducible and special
(see~\cite[\S6.1]{BK}). Namely, we have the natural embedding $P(\CO)_n\hookrightarrow G'(\CO)_{n+\frac12}$, and the
restriction of $\pi_\theta$ to $P(\CO)_n$ remains irreducible and is naturally isomorphic to $\ind_{P(\CO)_{2m-1}(m)}^{P(\CO)_{2m-1}}(\varXi)$
(resp.\ $\ind_{P'(\CO)_{2m}(m)}^{P(\CO)_{2m}}(\varXi)$) when $n=2m-1$ (resp.\ $n=2m$). Hence this restriction is
isomorphic to $\sigma_n=\ind_{\CO_n}^{P(\CO)_n}(\psi_n)$, see~Lemmas~\ref{trivial},\ref{intertwiners}. Here the character $\varXi$
is defined as follows.

Suppose first that $n=2m-1$ is odd. Just as in~\S\ref{odd depth},
the abelian subgroup $P(\CO)_{2m-1}(m)\subset P(\CO)_{2m-1}$
has the character (also denoted by $\theta\circ p$) given by the following composition:
\begin{equation}
  \label{theta}
  P(\CO)_{2m-1}(m)\hookrightarrow G'(\CO)_{2m-\frac12}(m)\xrightarrow{p}T^{\CE}_{2m-\frac12}(m)\xrightarrow{\theta}\BC^\times.
  \end{equation}
From the explicit formula for $p$ it is clear that $\theta\circ p$ factors through the projection to the strictly upper triangular
matrices $\CO_{2m-1}(m)\subset P(\CO)_{2m-1}(m)$: 
\begin{equation}
  \label{Psi}
  P(\CO)_{2m-1}(m)\twoheadrightarrow\CO_{2m-1}(m)\xrightarrow{\varXi}\BC^\times,
  \end{equation}
where $\varXi$ is defined as follows. For an element $ b=b_mt^m+\ldots+b_{2m-1}t^{2m-1}\in\CO_{2m-1}(m)$,
we have $\varXi( b)=\theta(1+\frac{\sqrt{-t}}{2} b)$. 
By an abuse of notation, we will denote the composition~\eqref{Psi} by $\varXi$ as well.
      
We can also consider the strictly upper triangular subgroup $\CO_{2m-1}\subset P(\CO)_{2m-1}$, and its character $\varPsi$ given by the formula
\begin{equation}
  \label{Xi}
  \varPsi(b_0+b_1t+\ldots+b_mt^m+\ldots b_{2m-1}t^{2m-1})=\theta(1+\frac{\sqrt{-t}}{2}(b_mt^m+\ldots+b_{2m-1}t^{2m-1}))
\end{equation}
(so that the characters~\eqref{Psi} and~\eqref{Xi} coincide on the intersection of $P(\CO)_{2m-1}(m)$ and $\CO_{2m-1}$).
In other words, for $ a\in\CO_{m-1}$,
\begin{equation}
  \label{ctheta}
  \varPsi(t^m a)=\theta(1+\frac{\sqrt{-t}t^m}{2} a)=\psi_{m-1}(c_\theta\cdot a),
\end{equation}
for $c_\theta\in\CO_{m-1}^\times$.

In the even case $n=2m$, the subgroup $P'(\CO)_{2m}(m)$ has the character given by
the composition 
\begin{equation}
  \label{theta even}
  P'(\CO)_{2m}(m)\hookrightarrow G'(\CO)_{2m+\frac12}(m+\frac12)\xrightarrow{p}T^{\CE}_{2m+\frac12}(m+\frac12)\xrightarrow{\theta}\BC^\times.
  \end{equation}
Again $\theta\circ p$ factors through the projection to the strictly upper triangular
matrices $\CO_{2m}(m)\subset P'(\CO)_{2m}(m)$: 
\begin{equation}
  \label{Psi even}
  P'(\CO)_{2m}(m)\twoheadrightarrow\CO_{2m}(m)\xrightarrow{\varXi}\BC^\times,
  \end{equation}
where $\varXi$ is defined as follows. For an element $ b=b_mt^m+\ldots+b_{2m}t^{2m}\in\CO_{2m}(m)$, we have
$\varXi( b)=\theta(1+\frac{\sqrt{-t}}{2} b)$.
By an abuse of notation, we will denote the composition~\eqref{Psi even} by $\varXi$ as well.

We can also consider the strictly upper triangular subgroup $\CO_{2m}\subset P(\CO)_{2m}$, and its character $\varPsi$ given by the formula
\begin{equation}
  \label{Xi even}
  \varPsi(b_0+b_1t+\ldots+b_mt^m+\ldots+b_{2m}t^{2m})=\theta(1+\frac{\sqrt{-t}}{2}(b_mt^m+\ldots+b_{2m}t^{2m}))
\end{equation}
(so that the characters~\eqref{Psi even} and~\eqref{Xi even} coincide on the intersection of $P'(\CO)_{2m}(m)$ and $\CO_{2m}$).
Similarly to~\eqref{ctheta}, we have $\varPsi(t^{m+1} a)=\theta(1+\frac{\sqrt{-t}t^{m+1}}{2} a)=\psi_{m-1}(c_\theta\cdot a)$.

\subsection{$\Gamma$-function, ramified}
As in~\S\ref{gamma unr}, given a multiplicative character $\omega\colon\CO^\times_n\to\BC^\times$, we may view it as a function
$\omega$ in the full conjugate dual space $\ind_{\CO_n}^{P(\CO)_n}(\psi_n)^\wedge$ of $\sigma_n=\ind_{\CO_n}^{P(\CO)_n}(\psi_n)$. We apply
$\pi_\theta^\wedge\begin{pmatrix}0&1\\ -t&0\end{pmatrix}$ to $\omega\in \ind_{\CO_n}^{P(\CO)_n}(\psi_n)^\wedge=\Res_{G'(\CO)_{n+\frac12}}^{P(\CO)_n}\pi_\theta^\wedge$.
We get $\omega^{-1}$ up to a constant multiple $\Gamma(\pi_\theta,\omega;\psi_n)\in\BC^\times$.
Recall also the notation introduced right before~Lemma~\ref{trivial}: $c\in\CO^\times_n$ is defined via the
identity $\varPsi(y)=\psi_n(cy)$ for any $y\in\CO_n$.

For $n>0$ we make the character parameters explicit. If $n=2m-1$, then
\[
\omega(1+t^m a)=\psi_{m-1}(c_\omega a),\qquad
\varPsi(t^m a)=\psi_{m-1}(c_\theta a),\qquad a\in\CO_{m-1}.
\]
If $n=2m>0$, then
\[
\omega(1+t^{m+1}a)=\psi_{m-1}(c_\omega a),\qquad
\varPsi(t^{m+1}a)=\psi_{m-1}(c_\theta a),\qquad a\in\CO_{m-1}.
\]
In both cases $c_\theta\in\CO_{m-1}^\times$ and $c_\omega\in\CO_{m-1}$.
In the following formula, $c_\omega/c_\theta$ is lifted to $\CO_n$
as its unique polynomial representative of degree $<m$.

\begin{thm}
  \label{compute gamma t}
  \textup{(1)} For $n=0$, we have $\Gamma(\pi_\theta,\omega;\psi_n)=\omega(c)^2\theta(\sqrt{-t})$.

  \textup{(2)} For $n>0$, we have $\displaystyle{\Gamma(\pi_\theta,\omega;\psi_n)
  =\theta(\sqrt{-t})\theta\left(1+\frac{c_\omega}{c_\theta}\sqrt{-t}\right)\omega\left(1+\frac{c_\omega^2}{c_\theta^2}t\right)}\omega(c)^2$. 
\end{thm}

\begin{proof}
  (1) is immediate. The proof of (2) is the same as of~Theorem~\ref{compute gamma eps}(1) (for $n$ odd) and the beginning of proof
  of~\ref{compute gamma eps}(3) (for $n$ even), making use of the identity
  \[\begin{pmatrix}x&z\\ 0&1\end{pmatrix}\begin{pmatrix}0&1\\ -t&0\end{pmatrix}=
  \begin{pmatrix}\frac{-xzt}{1+z^2t} & \frac{x}{1+z^2t}\\ \frac{-xt}{1+z^2t} & \frac{-xzt}{1+z^2t}\end{pmatrix}
  \begin{pmatrix}\frac{1+z^2t}{x} & -z\\ 0&1\end{pmatrix}.\] 
\end{proof}

\subsection{1-dimensional analogues}
We return to the setup of~\S\ref{field F} and assume that $\bF$ has finite odd characteristic: $\bF=\bk(\!(\varpi)\!)$.
We also assume that the character $\bar\psi_n\colon\bO_n\to\BC^\times$ of~\S\ref{setup} is the composition of $\bar\psi\colon\bk\to\BC^\times$
with taking the $n$-th coefficient of a Taylor series.
Then all the constructions and results of the present section have direct analogues for representations of $G(\bO)$ in
place of $G(\CO)$. We start with~Lemma~\ref{RRA}. So $\bl\simeq\BF_{q^2}$ is a quadratic extension of $\bk\simeq\BF_q$,
and $\bar\epsilon\in\bk$ is an element such that $\sqrt{\bar\epsilon}$ is a generator of $\bl$. We fix $\bar{r}\in\bk^\times$
and consider the Weil representation $\CW_{\psi_{\bar r}}$ of $\Sp(2,\bk)$ in $\BC[\bk]$ arising from the Heisenberg action in
$\BC[\bk]$ with central character $\psi_{\bar r}(x)=\bar\psi({\bar r}x)$ just as in the discussion preceding~Lemma~\ref{RRA}.
%Over the finite field $\bk$ we use the linear Weil representation of $\Sp(2,\bk)$ on $\BC[\bk]$.
Thus it is possible to write down explicit formulas for the action
of $\Sp(2,\bk)$ in $\BC[\bk]$, in particular for the action of the unit norm torus $\bl^1\subset\Sp(2,\bk)$. To this end we denote by $\chi\colon\bk^\times\twoheadrightarrow \bk^\times/(\bk^\times)^2\iso\{\pm1\}$ the unique nontrivial quadratic
character. For $a\in\bk^\times$, the absolute Weil index is normalized by
\[
 \gamma_\bk\!(y\mapsto\bar\psi(ay^2/2))
 =q^{-1/2}\sum_{x\in\bk}\bar\psi(ax^2/2).
\]
The following lemma is proved in~\cite[Appendix A.3.2]{GH}.

\begin{lem}
  \label{Gerardin}
An element $\begin{pmatrix}u&\bar\epsilon v\\ v&u\end{pmatrix}$ of $\bl^1$ acts
by the following automorphism $A_{u,v}$ of $\BC[\bk]$ when $v\ne0$:
\[A_{u,v}\phi(\xi)=\gamma_\bk\!\left(y\mapsto\bar\psi\left(-\frac{\bar{r}vy^2}{2\bar\epsilon}\right)\right)^{-1}\sqrt{q}^{-1}
\sum_{\eta\in\bk}\bar\psi\left(-\frac{{\bar r}(u\xi^2+u\eta^2-2\xi\eta)}{2\bar\epsilon v}\right)\phi(\eta).\]
Furthermore, for $v=0$, $u=1$ we have $A_{1,0}\phi(\xi)=\phi(\xi)$, and for
$v=0$, $u=-1$ we have $A_{-1,0}\phi(\xi)=\chi(-1)\phi(-\xi)$.
\end{lem}

\begin{proof}
Write $G(a)=\gamma_\bk(y\mapsto\bar\psi(ay^2/2))$. The finite-field
Gauss-sum identities are
\[
 G(a)=\chi(a)G(1),\qquad G(-a)=G(a)^{-1},\qquad G(a)^2=\chi(-1).
\]
For torus elements $g,h$ with $v_g,v_h,v_{gh}\ne0$, put
$B_i=\bar\epsilon v_i$. Summation in the intermediate variable in the
product of the displayed kernels gives the kernel for $gh$, with scalar ratio
\[
 \frac{G(-\bar rB_{gh})G(-\bar rB_{gh}/(B_gB_h))}
 {G(-\bar rB_g)G(-\bar rB_h)}=1.
\]
If $h=g^{-1}$, the intermediate sum is a delta function and the
product is the identity. If $h=-g^{-1}$, the product is
$G(-\bar rB_g)^{-2}P=\chi(-1)P$, where $P\phi(\xi)=\phi(-\xi)$.
Finally $A_{-g}=\chi(-1)PA_g$ and $A_{-1}^2=1$.
These calculations verify multiplicativity, including all exceptional cases;
covariance with the Heisenberg action follows by substitution in the kernel.
\end{proof}

Now transferring to our present setup the notation introduced and recalled before~Theorem~\ref{compute gamma eps} (like $c_\omega\leadsto c_{\bar\omega}$, $c_0\leadsto c_0$, etc.), the same argument as in the proof of~Theorems~\ref{compute gamma eps},\ref{compute gamma t} establishes the following

\begin{prop}
  \label{Adam and Eve}
  \textup{(1)} If $\theta$ is a non-self-conjugate depth $n=2m-1$ character
  of $T^\bE$ for an unramified quadratic extension $\bE$ of $\bF$, then
  \[\Gamma(\pi_{\theta},\bar\omega;\bar\psi_n)=\theta\left(\frac{c_{\bar\omega}}
  {c_{\theta}}+\sqrt{\bar\epsilon}^{-1}\right)
  \bar\omega\left(\frac{c^2_{\bar\omega}}
  {c^2_{\theta}}-\bar\epsilon^{-1}\right)\bar\omega(c)^2.\]
  \textup{(2)} If $\theta$ is a non-self-conjugate depth~$0$ character
  of $T^\bE$ for an unramified quadratic extension $\bE$ of $\bF$, then $\theta$ is the inflation
  of a non-self-conjugate character $\vartheta$ of $T^\bl$, and $\bar\omega$ is the inflation of
  the same named character of $\bk^\times$, while $\pi_\theta$ is the inflation of a cuspidal
  representation $\pi_\vartheta$ of $G(\bk)$. Then $\Gamma(\pi_{\theta},\bar\omega;\bar\psi_0)$
  is nothing but the finite field $\Gamma$-function $\Gamma(\pi_\vartheta,\bar\omega;\bar\psi)$. 
  
  \textup{(3)} If $\theta$ is a non-self-conjugate depth $n=2m>0$ character
  of $T^\bE$ for an unramified quadratic extension $\bE$ of $\bF$, then for $c_0\ne0$
  \begin{multline*}
  \Gamma(\pi_{\theta},\bar\omega;\bar\psi_n)=\theta\left(\frac{c_{\bar\omega}}
  {c_{\theta}}+\sqrt{\bar\epsilon}^{-1}\right)
  \bar\omega\left(\frac{c^2_{\bar\omega}}
  {c^2_{\theta}}-\bar\epsilon^{-1}\right)\bar\omega(c)^2\\
  \times
  \gamma_\bk\!\left(y\mapsto\bar\psi\left(\frac{c_{\theta}(0)c_0y^2}{(\bar\epsilon^2c_0^2-\bar\epsilon)}\right)\right)^{-2}
  \bar\psi\left(\frac{(\ell-b)^2c_0}{c_{\theta}(0)(c_0^2-\bar\epsilon^{-1})}\right),
  \end{multline*}               
  and for $c_0=0$, \[\Gamma(\pi_{\theta},\bar\omega;\bar\psi_n)
  =\theta\left(\frac{c_{\bar\omega}}
  {c_{\theta}}+\sqrt{\bar\epsilon}^{-1}\right)
  \bar\omega\left(\frac{c^2_{\bar\omega}}
  {c^2_{\theta}}-\bar\epsilon^{-1}\right)\bar\omega(c)^2
  \chi(-1).\]
  \textup{(4)} If $\theta$ is a non-self-conjugate depth~$\frac12$ character of
  $T^\bE$ for a ramified quadratic extension $\bE=\bF(\sqrt{-\varpi})$ of $\bF$,
  then \[\Gamma(\pi_{\theta},\bar\omega;\bar\psi_n)=\bar\omega(c)^2\theta(\sqrt{-\varpi}).\]
  \textup{(5)} If $\theta$ is a non-self-conjugate depth $n+\frac12>\frac12$ character of
  $T^\bE$ for a ramified quadratic extension $\bE=\bF(\sqrt{-\varpi})$ of $\bF$,
  then \[\Gamma(\pi_{\theta},\bar\omega;\bar\psi_n)=
  \theta(\sqrt{-\varpi})\theta\left(1+\frac{c_{\bar\omega}}{c_{\theta}}\sqrt{-\varpi}\right)
  \bar\omega\left(1+\frac{c_{\bar\omega}^2}{c_{\theta}^2}\varpi\right)
  \bar\omega(c)^2.\]
  \hfill $\Box$
\end{prop}

\begin{rem}
  Recall that by~Lemma~\ref{GamGam}, $\Gamma(\pi_{\theta},\bar\omega;\bar\psi_n)=\Gamma(\Pi,\omega;\psi)$.
  Under the Langlands correspondence, the $\Gamma$-function $\Gamma(\Pi,\omega;\psi)$ equals
  the $\varepsilon$-factor $\varepsilon(\frac12,V\otimes\omega^{-1};\psi)$, where
  $V$ is the 2-dimensional representation of the Weil group of $\bF$
  corresponding to $\Pi$ (and $\omega$ is viewed as the character
  of the Weil group under the local class field theory). Thus~Proposition~\ref{Adam and Eve} is essentially
  equivalent to the calculations of $\varepsilon$-factors in~\cite[\S\S23.6,25.5]{BH}.
\end{rem}

\subsection{Modifications for $\GL(2)$}
We follow~\cite[\S13]{BK}. We will consider special representations of $\CK^\times\cdot\GL(2,\CO)$ or $\CK^\times\cdot\GL'(2,\CO)$ 
(where the subgroup $\GL'(2,\CO)$
of $\GL(2,\CK)$ is generated by the Iwahori subgroup and $\begin{pmatrix}0&1\\ -t&0\end{pmatrix}$). The meaning of {\em special}
is the same as before: the restriction to $P(\CO)$ (viewed now as a subgroup of $\GL(2,\CO)$ or $\GL'(2,\CO)$) is an
irreducible representation of dimension $>1$. However, this restriction is not necessarily isomorphic to $\sigma_n$
(see~\S\ref{Rep PO} above~Lemma~\ref{trivial}) anymore.
It is isomorphic to the tensor product $\sigma_n\otimes\lambda$, where $\lambda$ is a character
$P(\CO)\twoheadrightarrow\CO^\times\twoheadrightarrow\CO^\times_k\to\BC^\times$ nontrivial on $\CO^\times_k(k)$ for some $k\geq n$.
Thus $n$ is the maximal among integers $m$ such that the restriction of our representation to
$\left\{\begin{pmatrix}1&t^m\CO\\ 0&1\end{pmatrix}\right\}$ is nontrivial. We will say that a representation of
$\CK^\times\cdot\GL(2,\CO)$ (resp.\ $\CK^\times\cdot\GL'(2,\CO)$) whose restriction to $P(\CO)$ is of the form $\sigma_n\otimes\lambda$
has {\em mirabolic depth} $n$ (resp. $n+\frac12$).

As before, any special representation is of the form $\pi_\theta$ for a
character $\theta\colon\CE^\times\to\BC^\times$ such that $\theta\ne\bar\theta$ (but we do not assume anymore that
$\theta$ factors through $T^\CE$). The central character of $\pi_\theta$ is $\mu\colon\CK^\times\to\BC^\times$, where
$\mu=\theta|_{\CK^\times}$. Let $n$ (resp.\ $n+\frac12$) be the mirabolic depth of $\pi_\theta$ in the unramified case (resp.\ in the ramified
case). Given a multiplicative character $\omega\colon\CO^\times\to\BC^\times$ such that
$(\omega\cdot\lambda^{-1})|_{1+t^{n+1}\CO}\equiv1$, the function $f_\omega(x):=(\omega\cdot\lambda^{-1})(x)$ viewed as a distribution in
$\CS'(\CO^\times_n)$, is a unique (up to multiplication with a nonzero scalar) vector of diagonal character $\omega$, i.e.\ such that
$\pi_\theta^\wedge\begin{pmatrix}a&0\\ 0&1\end{pmatrix}f_\omega=\omega(a)f_\omega$ (the meaning of $\lambda$ is the same as in the previous
paragraph). So now
$\pi_\theta^\wedge\begin{pmatrix}0&1\\ -1&0\end{pmatrix}f_\omega=\Gamma(\pi_\theta,\omega;\psi_n)f_{\mu\omega^{-1}}$ in the unramified case, and
$\pi_\theta^\wedge\begin{pmatrix}0&1\\ -t&0\end{pmatrix}f_\omega=\Gamma(\pi_\theta,\omega;\psi_n)f_{\mu\omega^{-1}}$ in the ramified case.
(Here we assume that $\pi_\theta$ is unitary since $\pi_\theta^\wedge$ is defined in the full {\em conjugate} dual.)
Note that the depth of $\theta$ may be bigger than the mirabolic depth of $\pi_\theta$.

In order to compute $\Gamma(\pi_\theta,\omega;\psi_n)$, we use the trick of~\cite[\S13.5]{BK}. Namely, for a unitary character
$\chi\colon\CK^\times\to\BC^\times$ such that $\chi(t)=1$, we have
$\Gamma(\pi_\theta,\omega;\psi_n)=\Gamma(\pi_\theta\otimes\chi(\det),\omega\cdot\chi;\psi_n)$,
and $\pi_\theta\otimes\chi(\det)=\pi_{\theta'}$, where $\theta'=\theta\cdot\chi(\on{Norm}_{\CE/\CK})$
(hence $\mu'=\mu\cdot\chi^2$ and $\lambda'=\lambda\cdot\chi$). Now it is possible to find a unitary character $\chi$
such that $\theta'|_{1+t\CO}\equiv1$ and $\mu'|_{1+t\CO}\equiv1$, and $\chi(t)=1$. Indeed, let $U=1+t\CO$, so that
$\CK^\times=t^\BZ\times\bF^\times\times U$. Since the residue characteristic of $\bF$ is odd, each element $u\in U$ has a unique square
root $\sqrt{u}$. We set $\chi(u)=\mu(\sqrt{u})^{-1}$, $\chi|_{\bF^\times}\equiv1$, $\chi(t)=1$.

We set $\omega'=\omega\cdot\chi$.
Note that the mirabolic depth of $\pi_{\theta'}$ is equal the mirabolic depth of $\pi_\theta$ (i.e.\ to $n$ in the unramified case,
and to $n+\frac12$ in the ramified case).
In other words, $\pi_{\theta'}|_{P(\CO)}\simeq\sigma_n\otimes\lambda'$. But the depth of $\omega'$ can change, so we must require that
$\omega'|_{1+t^{n+1}\CO}\equiv1$.
Under this assumption, the definitions
of~\S\S\ref{odd depth},\ref{even depth},\ref{Rep G'O} make sense in the present situation (for $\theta',\omega'$ in place of $\theta,\omega$).
Moreover, the proof of~Theorems~\ref{compute gamma eps},\ref{compute gamma t} goes through with two small caveats.
First, $\Gamma(\pi_{\theta'},\omega';\psi_n)=\frac{\omega'(c)^2}{\mu'(c)}\Gamma(\pi_{\theta'},\omega';\varPsi)$.
Second, right below~\eqref{identity} we cannot anymore replace
$\theta'(\frac{-xz}{z^2-\epsilon^{-1}}+\frac{-x\epsilon^{-1}}{z^2-\epsilon^{-1}}\sqrt{\epsilon})$
with $\theta'(z+\sqrt{\epsilon}^{-1})$ since $\theta'$ does not factor through $T^\CE$.
Hence in the formulas of~Theorem~\ref{compute gamma eps},
$\omega(c)^2$ will be replaced by $\frac{\omega'(c)^2}{\mu'(c)}$, and 
$\displaystyle{\theta\left(\frac{c_\omega}{c_\theta}+\sqrt{\epsilon}^{-1}\right)}$
will be replaced by $\displaystyle{\theta'\left(\frac{-\frac{c_{\omega'}}{c_{\theta'}}}{\frac{c_{\omega'}^2}{c_{\theta'}^2}-\epsilon^{-1}}+
\frac{-\epsilon^{-1}}{\frac{c_{\omega'}^2}{c_{\theta'}^2}-\epsilon^{-1}}\sqrt{\epsilon}\right)}$, while in the formula
of~Theorem~\ref{compute gamma t}, $\displaystyle{\theta(\sqrt{-t})\theta\left(1+\frac{c_\omega}{c_\theta}\sqrt{-t}\right)}$ will be
replaced by $\displaystyle{\theta'\left(\frac{-\frac{c_{\omega'}t}{c_{\theta'}}}{1+\frac{c_{\omega'}^2t}{c_{\theta'}^2}}+
\frac{\sqrt{-t}}{1+\frac{c_{\omega'}^2t}{c_{\theta'}^2}}\right)}$.

\end{document}